\documentclass{amsart}[11pt]
\usepackage{amsmath, amsfonts, amsthm, amssymb,mathtools}
\usepackage{subfigure}
\usepackage{stmaryrd}
\usepackage{verbatim}
\usepackage{hyperref}
\usepackage{color}
\usepackage{ulem}
\usepackage[dvips]{graphicx}
\usepackage{enumerate}
\numberwithin{equation}{section}
\newtheorem{theorem}{Theorem}[section]
\newtheorem{corollary}[theorem]{Corollary}
\newtheorem{lemma}[theorem]{Lemma}
\newtheorem{proposition}[theorem]{Proposition}

\theoremstyle{definition}
\newtheorem{definition}[theorem]{Definition}
\newtheorem{remark}[theorem]{Remark}

\newtheorem{assumption}[theorem]{Assumption}

\newcommand{\RR}{\mathbb{R}}

\newcommand{\KK}{\mathbb{K}}

\newcommand{\PP}{\mathbb{P}}

\newcommand{\ak}{\mathfrak{a}}
\newcommand{\bk}{\mathfrak{b}}
\newcommand{\ck}{\mathfrak{c}}

\newcommand{\D}{\mathrm{d}}

\newcommand{\Aa}{\mathcal{A}}
\newcommand{\Bb}{\mathcal{B}}
\newcommand{\Cc}{\mathcal{C}}

\newcommand{\Ee}{\mathcal{E}}
\newcommand{\Jj}{\mathcal{J}}
\newcommand{\Ss}{\mathcal{S}}
\newcommand{\Tt}{\mathcal{T}}
\newcommand{\Xx}{\mathcal{X}}

\newcommand{\Fh}{F^{H}}
\newcommand{\II}{\mathtt{I}}
\newcommand{\E}{\mathrm{e}}
\newcommand{\W}{\mathrm{W}}
\newcommand{\X}{\mathrm{X}}
\newcommand{\Y}{\mathrm{Y}}
\newcommand{\Z}{\mathrm{Z}}
\newcommand{\xx}{\mathrm{x}}
\newcommand{\yy}{\mathrm{y}}
\newcommand{\zz}{\mathrm{z}}

\newcommand{\BV}{\mathrm{BV}}

\newcommand{\LTt}{\mathrm{L}^2(\Tt)}

\newcommand{\Hh}{\mathcal{H}}
\newcommand{\Nn}{\mathcal{N}}

\newcommand{\Ff}{\mathcal{F}}
\newcommand{\Oo}{\mathcal{O}}
\newcommand{\eps}{\varepsilon}

\newcommand{\Hm}{H_{-}}
\newcommand{\Hp}{H_{+}}
\newcommand{\Abt}{\mathcal{A}_{b}^{\Theta}}

\newcommand{\supp}{\mathrm{supp}}
\newcommand{\LDP}{\mathrm{LDP}}

\numberwithin{equation}{section}
\numberwithin{equation}{section}
\numberwithin{equation}{section}

\begin{document}
\title{Asymptotic behaviour of randomised fractional volatility models}

\author{Blanka Horvath}
\address{Department of Mathematics, King's College London }
\email{blanka.horvath@kcl.ac.uk}

\author{Antoine Jacquier}
\address{Department of Mathematics, Imperial College London}
\email{a.jacquier@imperial.ac.uk}

\author{Chlo\'e Lacombe}
\address{Department of Mathematics, Imperial College London}
\email{chloe.lacombe14@imperial.ac.uk}

\date{\today}
\thanks{The authors are grateful to the anonymous referees for their useful comments.
BH acknowledges financial support from the SNSF Early Postdoc.Mobility grant 165248.}
\maketitle

\begin{abstract}
We study the asymptotic behaviour of a class of small-noise diffusions driven by fractional Brownian motion, 
with random starting points.
Different scalings allow for different asymptotic properties of the process 
(small-time and tail behaviours in particular).
In order to do so, we extend some results on sample path large deviations for such diffusions.
As an application, we show how these results characterise the small-time and tail estimates
of the implied volatility for rough volatility models, recently proposed in mathematical finance.
\end{abstract}

%%%%%%%%%%%%%%%%%%%%%%%%%%%%%%%%%%%%%%%%%%%%%%%%%%
%%%%%%%%%%%%%%%%%%%%%%%%%%%%%%%%%%%%%%%%%%%%%%%%%%
\section{Introduction}
Large deviations are used extensively in Physics (thermodynamics, statistical mechanics) as well as in Mathematics (information theory, stochastic analysis, mathematical finance) 
to estimate the exponential decay of probability measures of rare events. 
Varadhan~\cite{Varadhan}, Schilder~\cite{Schilder},  Freidlin and Wentzell~\cite{FWBook} proved, 
in different degrees of generality, large deviations principles (in~$\RR^n$ and on path space) 
for solutions of stochastic differential equations with small noise,
and the monographs by Dembo and Zeitouni~\cite{DZ} and Deuschel-Stroock~\cite{DBookS} 
provide a precise account of those advances (at least up to the mid-1990s).
In the past decade, this set of techniques and results has been adopted by the mathematical finance community:
finite-dimensional large deviations (in the sense of G\"artner-Ellis) have been used 
to prove small-and large-time asymptotics of implied volatility in affine models ~\cite{FJL, JKM},
sample-path LDP (\`a la Freidlin-Wentzell~\cite{FWBook}) have proved efficient to determine importance sampling changes of probability~\cite{GT16,GR08, Rob10},
and heat kernel expansions (following Ben Arous~\cite{Benarous} and Bismut~\cite{B84}),
have led to a general understanding of small-time and tail behaviour of 
multi-dimensional diffusions~\cite{BO1, BO2, DFJV1, DFJV2}.
These asymptotics have overall provided a deeper understanding of the behaviour of models,
and, ultimately, allow for better calibration of real data;
a general overview can be found in~\cite{Pham}.

Motivated by financial applications, we derive here asymptotic small-time and tail behaviours
of the solution to a generalised version of the Stein-Stein stochastic volatility model, 
originally proposed in~\cite{SZ99,SS}. 
We in particular consider two important (in light of the recent trends in the literature proposed models) extensions:
(i) the SDE driving the instantaneous volatility process is started from a random distribution;
this so-called `randomised' type of models was recently proposed in~\cite{JR16, JS17, M15}, 
in particular to understand the behaviour of the so-called `forward volatility';
(ii) the volatility process is driven by a fractional Brownian motion.
Fractional stochastic volatility models, originally proposed by Comte and Renault~\cite{ComteRenault, ComteCR} with $H\in (1/2,1)$, have recently been extended to the case $H \in (0,1/2)$,
and a recent flourishing activity in this area~\cite{ALV07, BFG16, Fukasawa, GJR14, JMM17}
has established these models as the go-to standards for estimation and calibration.

The original motivation behind randomisation of the initial starting point is rooted in financial practice, where only the initial value of the stock price process is observed directly and the instantaneous value of volatility is subject to calibration. 
The effect of randomisation of the initial volatility on the implied volatility surface was explored by Jacquier and Roome~\cite{JR15} in a simple `random environment' setting, where the volatility component was assumed to follow CEV dynamics. 
Their results give an impetus both on the theoretical and the practical level: they solve a practical modelling problem in a simple tractable setting and at the same time raise awareness for the potential prowess of applying random evolution equations for financial modelling. 
In this paper we follow up on this direction and blend more involved approaches (proposed by~\cite{M00, MNS}) from the literature around random environment and random evolution equations into our financial model, 
where randomness also appears in the drift and diffusion coefficients of the process.
On the practical level, independently from the results of~\cite{JR15}, 
Mechkov~\cite{M15} goes a step further in endorsing the idea of randomising the initial volatility and makes a strong case to move away from modelling hidden variables (such as stochastic volatility) in the traditional way. 
He argues that starting the volatility from a fixed starting point heavily underestimates the effect of the hidden variable on the slope of the implied volatility smile, 
and therefore `hot start' volatility models (with random starting point) significantly outperform traditional ones altogether. 
Indeed, both randomised models~\cite{JR15,M15} produce the desired explosion in the smile at short maturities. 
Jacquier and Shi~\cite{JS17} develop this further by providing a precise link between the rate of explosion of implied volatilities on the short end and the tail distribution of the initial distribution of the volatility process in a `randomised' Heston model. 
These outputs confirm that stochastic volatility models with random starting point constitute a class of counterexamples to the long-standing belief formulated by Gatheral~\cite[Chapter 5]{Gat06}, 
that jumps in the stock price process are needed to produce steep short-dated implied volatility skews. 
Another example of broadly different design was provided by Caravenna and Corbetta~\cite{CC17}. 
In their `multiscaling' model, the stock price process is continuous, 
while the volatility process has (carefully designed) jumps, 
and steepness of the smile is achieved with a heavy-tail distribution of the small-time distribution of the volatility.
Rough fractional volatility models (with continuous volatility paths) have recently been proposed, 
and are able to capture the volatility skew \cite{ALV07,BFG16, BFGHS17, EFR16, FZ16, Fukasawa, GJR14, GuliViens, JPS17}.
In this paper we analyse the combined effect of a rough fractional Brownian driver 
(with Hurst parameter $H\in (0,1)$) in the volatility and a random starting point. 
We quantify how the tail behaviour (parametrised by a scaling coefficient $b>0$) of the random starting point modulates the rate of explosion in the implied volatility in the presence of rough fractional volatility. 
Finally, in a specific simplified setting we highlight how our model blends naturally into the setting of forward-start options in stochastic volatility models, 
whose asymptotic properties have been studied in~\cite{JR16}.
In proving our results, we improve the large deviations literature on both SDEs with random starting points and fractional SDEs.

In Section~\ref{Sec:Notations}, we recall some concepts that will be used in the paper and set the notations.
We also introduce the model~\eqref{eq:model}, and the main assumptions on its dynamics and on the initial random starting point.
Section~\ref{Sec:Mainresults} collects the main large deviations estimates in different regimes: 
tail behaviour (Section~\ref{Sec:Tails}), and small-time behaviour (Section~\ref{Sec:Smalltime}). 
In each case, we present two different scenarios consisting of an appropriately rescaled fractional model (Theorems~\ref{Th:tails_frac_X} and~\ref{Th:LDP_ST_X}) 
and a simplified diffusive model (Theorems~\ref{thm:ext_MNSLDP} and~\ref{Th_rdm_new}) 
with more restrictive conditions on the random starting point (allowing for simpler large deviations rate functions). 
Section~\ref{sec:applications} displays applications to implied volatility asymptotics (Corollaries~\ref{Cor:ImpliedvolTails} and~\ref{Cor:ImpliedvolST}), 
and presents an application to forward-start options. 
Proofs can be found in Appendix.

%%%%%%%%%%%%%%%%%%%%%%%%%%%%%%%%%%%%%%%%%%%%%%%%%%%
%%%%%%%%%%%%%%%%%%%%%%%%%%%%%%%%%%%%%%%%%%%%%%%%%%%
\section{Set up and notations}\label{Sec:Notations}
As outlined in the introduction, we prove pathwise large deviations 
for a two-dimensional system generalising the Stein-Stein model~\cite{SZ99, SS}, with random initial datum.
In particular, via suitable rescaling, we determine the small-time and the large-tail behaviours of the system.
Before delving into the core of the paper, let us recall some useful facts about large deviations and Gaussian processes, which shall also serve as setting the notations for the rest of the paper. 
Unless otherwise stated, we always work on a finite time horizon, 
say $[0,1]$ without loss of generality, which we denote by~$\Tt$, and we write $\Tt^*:=\Tt\setminus\{0\}$.
We let $\Cc:=\Cc(\Tt, \RR)$ be the space of continuous functions from~$\Tt$ to~$\RR$
and $\Cc_b^2$ the space of twice differentiable functions on~$\Tt$ 
with bounded partial derivatives up to the second order.
We write $X^\eps\sim\LDP(h_\eps, I)$ when the sequence~$(X^\eps)_{\eps>0}$
satisfies a large deviations principle (Definition~\ref{def:LDP}) on~$\Cc$, 
as~$\eps$ tends to zero with good rate function~$I$ and speed~$h_{\eps}$, where~$h_{\eps}$ denotes a function satisfying $\lim_{\eps \downarrow 0} h_{\eps}=0$.
For a random variable~$X$, we denote by $\supp(X)$ its support.

\subsection{Large deviations and fractional Brownian motion}
We use~\cite{DZ} as our guide through large deviations.
Given a topological space~$\Xx$ and the completed Borel $\sigma$-field~$\Bb_{\Xx}$ corresponding to~$\Xx$, for any $A \in \Bb_{\Xx}$, we denote by~$\mathring{A}$ and~$\overline{A}$
respectively its interior and closure, and consider a sequence~$(X^\eps)_{\eps>0}$ on $(\Xx, \Bb_{\Xx})$. 

\begin{definition}
A (good) rate function is a lower semi-continuous mapping $I: \Xx \rightarrow [0, \infty]$ 
such that the level sets $\left\{ x: I(x) \le z \right\}$ are closed (compact) subsets of~$\Xx$ for any $z\geq 0$.
\end{definition}

\begin{definition}\label{def:LDP}
The sequence $(X^{\eps})_{\eps > 0}$ satisfies a large deviations principle (LDP)
on $(\Xx, \Bb_{\Xx})$ 
as~$\eps$ tends to zero,
with speed~$h_\eps$, and rate function~$I$, if for any Borel subset $A \subset \mathcal{X}$, 
the following inequalities hold:
\begin{equation}
-\inf_{A^o} \ I(\varphi) 
\le \liminf_{\eps \downarrow 0} h_\eps \log \mathbb{P} (X^{\eps} \in A) 
\le \limsup_{\eps \downarrow 0} h_\eps \log \mathbb{P} (X^{\eps} \in A)
\le - \inf_{\overline{A}} I(\varphi).
\end{equation}
\end{definition}
A particularly convenient tool to prove large deviations is the so-called exponential equivalence, 
which we recall from~\cite[Definition 4.2.10]{DZ} as follows:
\begin{definition}\label{def:ExpEquiv}
On a metric space $(\mathcal{Y},d)$, two $\mathcal{Y}$-valued sequences ${(X^{\eps})}_{\eps > 0}$ and 
${(\widetilde{X}^{\eps})}_{\eps > 0}$ are called exponentially equivalent 
(with speed~$h_\eps$) if there exist probability spaces $(\Omega, \Bb_\eps, \PP_\eps)_{\eps>0}$ such that
for any $\eps>0$, $\PP^{\eps}$ is the joint law and, 
for each $\delta >0$, the set $\left\{ \omega: (\widetilde{X}^{\eps},X^{\eps}) \in \Gamma_{\delta} \right\}$ is $\Bb_{\eps}$-measurable, and
\begin{equation*}
\limsup_{\eps \downarrow 0} h_\eps \log \PP^{\eps}\left(\Gamma_\delta\right)=- \infty,
\end{equation*}
where $\Gamma_\delta := \left\{ (\tilde{y},y): d(\tilde{y},y) > \delta \right\} \subset \mathcal{Y}\times\mathcal{Y}$.
\end{definition}
\begin{theorem}
Let ${(X^{\eps})}_{\eps > 0}$ and ${(\widetilde{X}^{\eps})}_{\eps > 0}$ be two 
exponentially equivalent sequences (with speed~$h_\eps$) on some metric space.
If $(X^{\eps})\sim\LDP(h_\eps, \Lambda^X)$ for some good rate function~$\Lambda^X$, 
then $(\widetilde{X}^{\eps})\sim\LDP(h_\eps, \Lambda^X)$.
\end{theorem}
The last tool we shall need repeatedly is the following Contraction Principle~\cite[Theorem 4.2.1]{DZ}:
\begin{theorem}[Contraction Principle]
Let $\mathcal{X}$, $\mathcal{Y}$ two topological spaces, $f: \mathcal{X} \rightarrow \mathcal{Y}$ a continuous function and $I: \mathcal{X} \rightarrow [0,\infty]$ a good rate function. 
For each $y\in\mathcal{Y}$, define $I'(y) := \inf \{I(x): x\in\mathcal{X}, y=f(x)\}$. Then, if $I$ controls the LDP associated with a family of probability measures $\{\mu_\eps\}$ on $\mathcal{X}$, 
then~$I'$ controls the LDP associated with the family of probability measures $\{\mu_\eps \circ f^{-1} \}$ on $\mathcal{Y}$ and $I'$ is a good rate function on $\mathcal{Y}$.
\end{theorem}

On a real, separable Banach space $(\Ee, \|\cdot\|)$,
we denote by $\Bb$ the associated Borel sigma field. 
Letting~$\Ee^*$ denote the topological dual of~$\Ee$, we define a Gaussian measure as follows:
\begin{definition}
A Gaussian measure~$\mu$ on $(\Ee, \|\cdot\|)$  is such that every $\vartheta^* \in\Ee^*$, 
when viewed as a random variable via the dual pairing
$\vartheta\mapsto\langle \vartheta^*, \vartheta\rangle_{\Ee^*\Ee}$,
is a real Gaussian random variable on $(\Ee, \Bb, \mu)$.	
\end{definition}
We associate a Gaussian process to a Gaussian measure in the usual way~\cite[Section 3.2]{CT06}.
Particular examples of Gaussian processes, crucial for the rest of the paper,
include standard Brownian motion on the time interval~$\Tt$,
where $\Ee=\Cc$ equipped with the supremum norm and with the topology of uniform convergence,
and~$\Ee^*$ is the space of signed measures on~$\Tt$.
In fact, this construction applies to all (centered) continuous Gaussian processes, 
which are uniquely characterised by their covariance operator.
A fractional Brownian motion $W^H$ with Hurst parameter $H \in (0,1)$ 
is such a Gaussian process, starting from zero, 
with covariance
$$
\left\langle W^H_t, W^H_s\right\rangle = \frac{1}{2} \left( {|t|}^{2H} + {|s|}^{2H} - {|t-s|}^{2H}\right), \qquad \textrm{for any } s,t \in \Tt.
$$
Of primary importance in understanding small-noise behaviours of Gaussian systems 
is the concept of reproducing kernel Hilbert spaces (RKHS), 
which we recall following~\cite[Definition 3.3]{CT06}:

\begin{definition}
Let~$\mu$ be a Gaussian measure on~$\Ee$ and define the map
$\mathcal{R} : \Ee^* \rightarrow \Ee$ by $\mathcal{R}x^* := \int_\Ee \langle x^*, x \rangle x \mu(\D x)$.
The RKHS~$\mathcal{H}_{\mu}$ of~$\mu$ is the completion of the image $\mathcal{R}\Ee^*$ for the norm
$\|\mathcal{R}x^*\|_{\mathcal{H}_{\mu}} := \left(\langle x^*, \mathcal{R}x^*\rangle\right)^{1/2}$,
for all $x^*\in\Ee^*$.
\end{definition}
To characterise the RKHS of fractional Brownian motion,
the usual tool is its Volterra representation~\cite{NVV99}
\begin{equation}\label{eq:VolterraW}
W^H_t = \int_0^t K^H(t,s) \D B_s,
\end{equation}
which holds almost surely for all $t\in\Tt$, 
where~$B$ is a standard Brownian motion generating the same filtration as~$W^H$,
and $K^H$ is the Volterra kernel defined, for any $s, t \in \Tt$ with $0<s<t$,
by~\cite[Theorem 5.2]{NVV99}
\begin{equation}\label{eq:VolterraK}
K^H(t,s) = 
\left\{
\begin{array}{ll}
\displaystyle \frac{\kappa_H}{s^{\Hm}}
\left[ \left(t(t-s)\right)^{\Hm} - \Hm
\int_{s}^{t} \frac{(u-s)^{\Hm}}{u^{1-\Hm}} \D u\right], & \displaystyle \text{if }H<\frac{1}{2},\\
\displaystyle \frac{\kappa_H \Hm}{s^{\Hm}}
\int_{s}^{t}\frac{u^{\Hm}\D u}{|u-s|^{1-\Hm}}, & \displaystyle \text{if }H>\frac{1}{2},\\
1, & \displaystyle \text{if }H=\frac{1}{2},\end{array}
\right.
\end{equation}
with $H_{\pm} := H\pm\frac{1}{2}$
and
$\displaystyle
\kappa_H := {\left(\frac{2H \Gamma(1-\Hm)}{\Gamma(\Hp)\Gamma(2-2H)}\right)}^{1/2}$.
For $t\in\Tt^*$, the map $K^H(t,\cdot)$ is square integrable around the origin,
and the reproducing kernel Hilbert space of the fractional Brownian motion is given by
$$
\Hh_{K^H} := \left\{ \int_{0}^{t} K^H(t,s) f(s) \D s, t\in\Tt : f \in \LTt\right\}
$$
with inner product
$$
\left\langle \int_{0}^{\cdot} K^H(\cdot,s) f_1(s) \D s, \int_{0}^{\cdot} K^H(\cdot,s) f_2(s) \D s \right\rangle_{\Hh_{K^H}} := {\langle f_1,f_2 \rangle }_{\LTt}.
$$
The notation~$\Hh_{K^H}$, emphasising the link with the underlying kernel, 
will be useful later (in Definition~\ref{def:RKHS}) for more general kernels.
In particular, the RKHS associated to (standard) Brownian motion ($H=1/2$)
is the Cameron-Martin space, and corresponds to the space of absolutely continuous functions starting at zero, 
with square integrable derivatives. 
In other words, for any $H \in (0,1)\setminus\{1/2\}$, 
the identity $\Hh_{K^{H}}=\KK_{K^H}\LTt$
holds, where $\KK_{K^H}: \LTt\ni f \mapsto \int_{0}^{t} K^H(t,s) f(s) \D s$.
This characterisation motivates the following definition:
\begin{definition}\label{def:RKHS}
For any strictly positive function~$\Phi:\RR_+^2\to\RR$ such that $\Phi(t, \cdot) \in L^2(\Tt)$
for any $t\in\Tt$, 
the corresponding RKHS is defined as
\begin{align}\label{eq:RKHSgeneral}
\Hh_{\Phi} := \left\{ \int_{0}^{t} \Phi(t,s) f(s) \D s, t\in\Tt:  f \in \LTt\right\},
\end{align}
with inner product
$$
\left\langle \int_{0}^{\cdot} \Phi(\cdot,s) f_1(s) \D s, \int_{0}^{\cdot} \Phi(\cdot,s) f_2(s) \D s \right\rangle_{\Hh_\Phi} := {\langle f_1,f_2 \rangle }_{\LTt}.
$$
\end{definition}

Reproducing Kernel Hilbert Spaces, together with their inner products, 
turn out to provide the right spaces to characterise large deviations rate functions.
In particular, for a given Gaussian Volterra process of the form 
$\int_{0}^{\cdot}\Phi(\cdot,s)\D B_s$, for some Volterra kernel~$\Phi$, 
it follows from~\cite[Theorem 3.4.5]{DBookS} 
that the sequence
$(\eps \int_{0}^{\cdot}\Phi(\cdot,s)\D B_s)_{\eps>0}$
satisfies a large deviations principle with speed~$\eps^2$ and rate function 
\begin{equation}\label{eq:LDPVolterrageneral}
\Lambda_{\Phi}(\varphi) = \left\{
\begin{array}{ll}
\displaystyle \frac{1}{2}\|\dot{\varphi}\|_{\Hh_{\Phi}}, & \text{if }\varphi \in \Hh_{\Phi},\\
+ \infty, & \text{otherwise}.
\end{array}
\right.
\end{equation}
An obviously special role is played by the standard Brownian motion $H=\frac{1}{2}$, 
and we shall adopt the simplified notation~$\Hh$ (the classical Cameron-Martin space) 
and~$\Lambda$ in place of $\Hh_{K^{1/2}}$ and $\Lambda_{K^{1/2}}$.

%%%%%%%%%%%%%%%%%%%%%%%%%%%%%%%%
%%%%%%%%%%%%%%%%%%%%%%%%%%%%%%%%
\subsection{Setting and assumptions}
The particular system we are interested in is 
\begin{equation}\label{eq:model}
\left\{
\begin{array}{ll}
\displaystyle
\D X_t = - \frac{1}{2}\sigma(Y_t)^2\D t + \sigma(Y_t) \left(\rho\D B_t + \overline{\rho} \D B^{\perp}_t\right), 
 & X_0 = 0, \\
\D Y_t = (\lambda + \beta Y_t) \D t + \xi \D W^H_t, & Y_0 \sim \Theta,
\end{array}
\right.
\end{equation}
where $W^H$ is a fractional Brownian motion, with Hurst parameter $H \in (0,1)$,
$(B,B^\perp)$ is a two-dimensional standard Brownian motion, 
$\beta<0$, $\lambda, \xi>0$, $\rho \in (-1,1)$, $\overline{\rho} := \sqrt{1-\rho^2}$,
and~$\Theta$ is a square-integrable continuous random variable.
In order to guarantee existence and uniqueness of a strong solution, 
we further assume~\cite[Theorem 3.1.3]{M08}
 that~$\sigma$ is Lipschitz continuous, 
 satisfies the growth condition $|\sigma(y)| \le C(1+|y|)$ for $y \in \RR$,
 and is differentiable with locally H\"older continuous derivative. 
In order to prove our main results below, we make the following technical assumption:

\noindent \textbf{Assumption $\Aa$}:
There exist a measurable function $\widetilde{\sigma}:\RR\to\RR$,
with locally H\"older continuous derivative, such that the scaling property $\lim_{\eps \downarrow 0} \eps \sigma(y/\eps) = \widetilde{\sigma}(y)$ holds true uniformly on~$\RR$.
\begin{remark}
Since the map~$\sigma$ is Lipschitz continuous with linear growth, 
then so is~$\widetilde{\sigma}$. 
The scaling property further implies that $\lim\limits_{\eps \downarrow 0} \eps^b \sigma(y/\eps^b) = \widetilde{\sigma}(y)$ uniformly on~$\RR$
for any $b>0$.
This will be useful for later computations.
\end{remark}

%%%%%%%%%%%%%%%%%%%%%%%%%%%%%%%%%%%%%
%%%%%%%%%%%%%%%%%%%%%%%%%%%%%%%%%%%%%
\section{Main results}\label{Sec:Mainresults}
Centrepiece of our analysis are large deviations estimates for suitably rescaled versions of~\eqref{eq:model}. 
The first rescaling (presented in Section~\ref{Sec:Tails}) is tailored 
to the analysis of the tail behaviour of~\eqref{eq:model}, 
while the second rescaling (Section~\ref{Sec:Smalltime}) is bespoke to its short-time asymptotic properties. 
In addition to these asymptotic results in the general fractional case (Theorems~\ref{Th:tails_frac_X} and~\ref{Th:LDP_ST_X}), 
we present two special simplified diffusive cases, 
where particularly tractable rate functions can be obtained (Theorems~\ref{thm:ext_MNSLDP} 
and~\ref{Th_rdm_new} respectively). 
For this we impose stronger conditions (Assumption~\ref{ass:MNS}) on the random starting point. 
This allows us to establish, following~\cite{MNS}  in Section~\ref{app:MNS}, an exponential equivalence 
between~\eqref{eq:model} and an analogous process with fixed starting point. 
In Section~\ref{Sec:Mellouk} we construct a third
rescaling~\eqref{eq:SystemEps} inspired by Mellouk~\cite{M00} in the short-time diffusive case under the assumption that the support of the random starting point is bounded. 
We shall work with rescaled versions $Y^\eps$ of the process~$Y$ in~~\eqref{eq:model}
(see~\eqref{eq:fracSSRandom} and~\eqref{eq:SmalltimeModel} for specific examples), together with a function~$h_\eps$ 
describing the speed of the large deviations estimates, for which we introduce the following assumptions:

\noindent \textbf{Assumption $\Aa'$}:
There exists a family of continuous functions $(\sigma_n)_{n \ge 0}$ on $\RR$ such that
\begin{itemize}
\item[(i)] $(\sigma_n)_{n\geq 0}$ converges uniformly to $\widetilde{\sigma}$ on~$\RR$;
\item[(ii)] for all $\delta >0$,
$\lim\limits_{n \uparrow \infty} \limsup\limits_{\eps \downarrow 0} h_{\eps} \log \PP \left(|\sigma_n(Y^\eps) - \widetilde{\sigma}(Y^\eps) \ge \delta| \right) = -\infty$.
\end{itemize}

\noindent \textbf{Assumption $\Abt$} (tail behaviour of~$\Theta$):
The limit
$\limsup\limits_{\eps \downarrow 0} h_\eps \log\PP(\eps^b |\Theta|>1) = -\infty$ holds.\\
In Assumptions $\Aa'$ and $\Abt$ above, the large deviations speed~$h_\eps$ 
takes the value~$\eps^{2b}$ in Section~\ref{Sec:Tails} 
and~$\eps^{4H+2b}$ in Section~\ref{Sec:Smalltime} respectively.
The constant~$b$ (which may vary below) 
plays an essential role in subsequent large deviations estimates,
via exponential equivalence techniques (Definition~\ref{def:ExpEquiv}).
Assumptions~$\Aa$ and~$\Aa'$ are naturally satisfied in the fractional Stein-Stein case 
(where $\sigma(y)\equiv y$), but imposing them allow us to state our results for more general~$\sigma$, 
in particular when using extended Contraction Principles~\cite[Proposition~2.3]{MNP92}.

%%%%%%%%%%%%%%%%%%%%%%%%%%%%%%%%%%%%%%%%%%%%%%
 \subsection{Tail behaviour}\label{Sec:Tails}
\subsubsection{The general case}
For $b,\eps>0$, introduce the rescaling 
$(X^{\eps}, Y^{\eps}) := (\eps^{2b}X, \eps^b Y)$, so that~\eqref{eq:model} becomes
\begin{equation}\label{eq:fracSSRandom}
\left\{
\begin{array}{ll}
\displaystyle
\D X^{\eps}_t = - \frac{\eps^{2b}}{2} \sigma\left(\frac{Y^{\eps}_t}{\eps^b}\right)^2 \D t + \eps^{2b} \sigma\left(\frac{Y^{\eps}_t}{\eps^b}\right) (\rho\D B_t + \overline{\rho} \D B^{\perp}_t), &X^{\eps}_0=0, \\
\D Y^{\eps}_t = \left( \eps^b \lambda + \beta Y^{\eps}_t \right) \D t + \eps^b \xi \D W^H_t, &Y^{\eps}_0 \sim \eps^b \Theta.
\end{array}
\right.
\end{equation}
The particular rescaling considered here is perfectly suited for tail behaviour,
as large deviations provide estimates for $\PP(X^\eps\geq 1) = \PP(X \geq \eps^{-{2b}})$.
Our main result is as follows, and is proved in Appendix~\ref{Sec:Proof1}:

\begin{theorem}\label{Th:tails_frac_X}
For any $H \in (0,1)$, the following hold:
\begin{enumerate}[(i)]
\item for any $b>0$ such that~$\Abt$ holds, $Y^{\eps} \sim \LDP \left(\eps^{2b}, \Lambda_{\Fh} \right)$,
with~$\Lambda_{\Fh}$ in~\eqref{eq:LDPVolterrageneral} 
and~$\Fh$ in  Lemma~\ref{lem:rep_fBm};
\item 
for any $b \ge \frac{1}{2}$ such that Assumptions~$\Aa$, $\Aa'$, $\Abt$ hold, 
$X^{\eps}\sim\LDP(\eps^{2b}, \widetilde{\Lambda})$, 
with~$\widetilde{\Lambda}$ in~\eqref{eq:grf_X_tails}. 
\end{enumerate}
\end{theorem}

The proof of the theorem, developed later, requires a precise analysis of 
the Reproducing Kernel Hilbert Spaces of the processes under consideration, 
and we first state two key ingredients
(proved in Appendices~\ref{App:Lem_F} and~\ref{sec:proof_RKHS_tails}), 
which are also of independent interest.
Recall from~\cite[Definition 2.1]{YLX08} the definition of the stochastic integral: for any smooth function~$f$ on $\Tt$ with bounded derivatives such that 
$\|f\|:= \int_{\Tt} {(\Gamma^{*}_{H,t} f(t))}^2 \D t < \infty$, define, for $t\in\Tt$,
$ \int_{0}^{t} f(u) \D W^H_u := \int_{0}^{t} \Gamma^{*}_{H,t} f(u) \D Z_u $, with 
$$
\Gamma^{*}_{H,t} f(s) := - \frac{\kappa_H}{s^{\Hm}} \frac{\D}{\D s} \int_{s}^{t} u^{\Hm} {(u-s)}^{\Hm} f(u) \D u,
\qquad\text{for all }s\in (0,t].
$$
\begin{lemma}\label{lem:rep_fBm}
For any $H\in (0,1)$ and $\beta>0$, there exists a standard Brownian motion~$Z$, 
such that
\begin{equation}\label{eq:RepfOU}
\xi \int_{0}^{t} \E^{\beta (t-s)} \D W^H_s = \int_{0}^{t} \Fh(t,s) \D Z_s,
\end{equation}
holds almost surely for $t\in\Tt$, where $\Fh:\Tt\times\Tt\to\RR$ is defined for $0<s<t$,
with~$\kappa_H$ in~\eqref{eq:VolterraK}, as
\begin{equation*}
\Fh(t,s) := \left\{
 \begin{array}{ll}
\displaystyle
\frac{\xi\kappa_H}{s^{\Hm}}\left[[t(t-s)]^{\Hm}
+ \int_{s}^{t}\left\{\frac{1-2H}{2u} + \beta \right\}
[u(u-s)]^{\Hm} \E^{\beta (t-u)} \D u \right],
 & \displaystyle\text{if }H <\frac{1}{2},\\
\displaystyle\frac{\xi\kappa_H \Hm}{s^{\Hm}} \int_s^t \frac{u^{\Hm}\E^{\beta (t-u)} \D u}{(u-s)^{1-\Hm}},
 & \displaystyle\text{if }H >\frac{1}{2},\\
\displaystyle\xi\E^{\beta (t-s)}, & \displaystyle\text{if } H = \frac{1}{2}.
\end{array}
\right.
\end{equation*}
\end{lemma}
The case $\beta=0$ (and $\xi =1$) is excluded since, in that case, 
the lemma boils down to the classical Volterra representation 
of the fractional Brownian motion~\eqref{eq:VolterraW} and the function~$\Fh$ is nothing else 
than the kernel~$K^H$ given in~\eqref{eq:VolterraK}.
For $H\in (\frac{1}{2},1)$, the expression for~$\Fh$ is in agreement with~\cite[Definition 2.1]{YLX08}, 
but, for $H \in(0,\frac{1}{2})$, 
it corrects the slightly erroneous expression therein.
This function~$F^H$ allows us to fully characterise the following RKHS,
with proof postponed to Section~\ref{sec:proof_RKHS_tails}:

\begin{proposition}\label{prop:RKHS_tails}
For any $H \in (0,1)$, $\beta>0$, the space~$\Hh_{\Fh}$
is the RKHS  of the Gaussian process 
$\xi\int_{0}^{\cdot} \E^{\beta (\cdot-s)} \D W^H_s$.
\end{proposition}

%%%%%%%%%%%%%%%%%%%%%%%%%%%%%%%%%%%%%%%%%%%%%%%%%%%
\subsubsection{The Millet-Nualart-Sanz approach}\label{app:MNS}

In~\cite{MNS}, Millet, Nualart and Sanz consider a perturbed stochastic differential equation of the form
\begin{equation}\label{eq:MNSsystem}
\D \X^{\eps}_t = \bk(\eps,\X^{\eps}_t)\D t + \eps \ak(\X^{\eps}_t)\D \W_t.
\end{equation}
Here, for any~$\eps>0$, the functions 
$\bk(\eps,\cdot):\RR^n\to\RR^n$ and $\ak: \RR^n \rightarrow \mathcal{M}_{(n,d)}(\RR)$
are bounded Borel measurable and uniformly Lipschitz,
$\bk(\eps,\cdot)$ converges uniformly to a function~$\bk(\cdot)$ as~$\eps$ tends to zero, 
$\W$ is a $d$-dimensional Brownian motion,
and~$\X^\eps_0$ is an $\RR^n$-valued square-integrable random variable.
Existence and uniqueness of a strong solution can be found in~\cite[Chapter 5, Theorem 2.1]{Friedman}.
Following classical large deviations steps, consider,
for any $\varphi\in\Hh$, the controlled ordinary differential equation on~$\Tt$:
\begin{equation}\label{eq:ControlODE}
\dot{\psi}_t = \ak(g_t) \dot{\varphi}_t + \bk(\psi_t),
\end{equation}
the solution flow of which, starting from~$\xx_0\in\RR^n$ is denoted by~$\Ss_{\xx_0}(\varphi)$.
Millet, Nualart and Sanz~\cite{MNS} proved a large deviations principle~\cite[Theorem 4.1]{MNS} 
for the sequence~$(\X^\eps)_{\eps>0}$ under the following assumption:
\begin{assumption}\label{ass:MNS}
Both~$\ak(\cdot)$ and~$\bk(\cdot)$ belong to~$\Cc_b^2$, and there exists $\xx_0 \in \RR^n$ such that, for any $\delta >0$,
\begin{equation}\label{assump1}
\limsup_{\varepsilon \downarrow 0} \eps^2 \log \PP \left(|\X^\eps_0-\xx_0| > \delta \right) = - \infty.
\end{equation}
\end{assumption}
\begin{theorem}\label{thm:MNSLDP}
Under Assumption~\ref{ass:MNS}, 
$(\X^{\eps})_{\eps >0}\sim\LDP(\eps^2, I)$ with
$I(\psi) = \inf\{\Lambda(\varphi): \varphi \in \Hh, \psi=\Ss_{\xx_0}(\varphi)\}$.
\end{theorem}

Condition~\eqref{assump1} is an exponential equivalence property 
between the initial random variable~$\X^\eps_0$
and the constant~$\xx_0$, and ensures that large deviations are preserved under exponentially small perturbations
of the starting point.
Therefore, in the standard diffusion case $H = \frac{1}{2}$, 
it is possible to obtain a similar result to Theorem~\ref{Th:tails_frac_X} with a simplified rate function 
(albeit with slightly more restrictions on the starting point), 
by using the approach considered by Millet-Nualart-Sanz in~\cite{MNS}.
For this, we rewrite~\eqref{eq:MNSsystem} to correspond to~\eqref{eq:fracSSRandom}, albeit with stronger assumptions on the coefficients, with 
$\W:=(W_1, W_2)'$ a Brownian motion, 
$\X^\eps = (X^\eps, Y^\eps)$, $H=\frac{1}{2}$,
$\eps\to\eps^b$,
$\X_0^\eps = (0, \eps \Theta)$, and $\overline{\rho} := \sqrt{1-\rho^2}$,

$$
\bk(\eps,\X^{\eps}_t) = 
\begin{pmatrix}
- \frac{1}{2} \widetilde{\sigma}(Y^{\eps}_t)^2\\
\eps \lambda + \beta Y^{\eps}_t
\end{pmatrix}
\qquad\text{and}\qquad
\ak(\X^{\eps}_t) = 
\begin{pmatrix}
\overline{\rho}\widetilde{\sigma}(Y^{\eps}_t) & \rho\widetilde{\sigma}(Y^{\eps}_t)\\
0 & \xi
\end{pmatrix}.
$$
The correlation between the two components of $\W$ is explicitly represented in the diffusion matrix~$a$.

\begin{theorem}~\label{thm:ext_MNSLDP}
Under Assumption~\ref{ass:MNS}, 
the solution~$\X^{\eps}$ to~\eqref{eq:MNSsystem}
satisfies $(\X^{\eps})_{\eps >0}\sim\LDP(\eps^2, \mathbf{I})$ with
$\mathbf{I}(\chi) = \inf \left\{ \Lambda(\varphi) : \varphi \in \Hh, \chi = \Ss_{\xx_0}(\Psi^{\rho}(\varphi))\right\}$
and
$\Psi^{\rho}:\RR^2\ni\mathrm{z}\mapsto 
\begin{pmatrix}
\overline{\rho} & \rho\\
0 & 1
\end{pmatrix}
\mathrm{z},
$.
\end{theorem} 

\begin{proof}
The proof of Theorem~\ref{thm:MNSLDP} relies 
first on proving a large deviations principle 
for the flow~$\Ss_{\xx_0}$ using Schilder's theorem~\cite[Theorem 5.2.3]{DZ}, 
then on extending this LDP to the original system.
One can easily extend it to include a correlation parameter $\rho \in (-1,1)$, 
the main difference being the rate function.
Indeed, since $W^\perp_2$ and $W_2$ are independent, 
Schilder's theorem yields that $\eps(W^\perp_2, W_2)' \sim \LDP(\eps^2, \Lambda)$. 
Since the map~$\Psi^{\rho}$ is continuous on $(\Cc, \|\cdot\|_{\infty})$ 
and $\eps\W' =\Psi^{\rho}(\eps(W^\perp_2, W_2)')$, 
the theorem follows from the Contraction Principle giving an LDP for~$\eps\W'$ 
as~$\eps$ tends to zero with speed~$\eps^2$ and good rate function 
\begin{equation}\label{eq:ext_grf}
\Lambda^{\rho}(\psi):=
\inf \left\{ \Lambda(\varphi): \varphi\in\Hh, \psi=\Psi^{\rho}(\varphi) \right\}.
\end{equation}
\end{proof}

%%%%%%%%%%%%%%%%%%%%%%%%%%%%%%%%%%
%%%%%%%%%%%%%%%%%%%%%%%%%%%%%%%%%%
\subsection{Small-time behaviour}\label{Sec:Smalltime}
We now tackle the small-time behaviour of the process~\eqref{eq:model}.
Under the general set of assumptions
$\Aa$, $\Aa'$, $\Abt$, we need to introduce a particular rescaling, 
both in time and in space in order to observe some weak convergence.
This is different from the classical It\^o diffusion case (with fixed starting point),
where solutions of such SDEs generally converge in small time.
In the It\^o case, though, if the distribution of the starting point has compact support,
we show in Section~\ref{Sec:Mellouk} that space rescaling is not required any longer.

%%%%%%%%%%%%%%%%%%%%%%%%%%%%%%%%%%%
\subsubsection{The general case}\label{sec:GenCase}
With the rescaling 
$(X^{\eps}_t, Y^{\eps}_t) := (\eps^{2H+2b-1} X_{\eps^2 t}, \eps^b Y_{\eps^2 t})$,
with $b>0$, \eqref{eq:model} becomes
\begin{equation}\label{eq:SmalltimeModel}
\left\{
\begin{array}{ll}
\displaystyle
\D X^{\eps}_t = - \frac{\eps^{2H+1+2b}}{2} \sigma\left(\frac{Y^{\eps}_t}{\eps^b}\right)^2 \D t + \eps^{2H+2b} \sigma\left(\frac{Y^{\eps}_t}{\eps^b}\right) \left(\rho\D B_t + \overline{\rho} \D B^{\perp}_t\right), &X^{\eps}_0=0, \\
\D Y^{\eps}_t = \left( \eps^{b+2} \lambda + \beta \eps^2 Y^{\eps}_t \right) \D t + \eps^{2H+b} \xi \D W^H_t, &Y^{\eps}_0 \sim \eps^b \Theta.
\end{array}
\right.
\end{equation}

Our main result is as follows:

\begin{theorem}\label{Th:LDP_ST_X}
For any $H \in (0,1)$, 
\begin{enumerate}[(i)]
\item for any $b>0$ such that~$\Abt$ holds
$Y^{\eps}\sim\LDP\left(\eps^{4H+2b}, \Lambda_{G^H_0}\right)$, with~$\Lambda_{G^H_0}$ 
as in~\eqref{eq:LDPVolterrageneral};
\item 
if $b\ge \frac{1}{2}-2H$ such that $\Aa$, $\Aa'$, $\Abt$ hold, 
$X^\eps\sim\LDP(\eps^{4H+2b}, \II)$, with~$\II$ defined in~\eqref{eq:grf_X_ST}.
\end{enumerate}
\end{theorem}

The proof of~(i) is similar to that of Theorem~\ref{Th:tails_frac_X}(i) 
and relies on proving LDP for an auxiliary process, defined in~\eqref{eq:auxi}, 
exponentially equivalent to the original (rescaled) process~$Y^{\eps}$.
The proof of~(ii) is more involved and postponed to Appendix~\ref{sec:Th:LDP_ST_X}.
In order to state the following key result, define, for any $\eps>0$, $G^H_{\eps}$ 
as the function~$\Fh$ in Lemma~\ref{lem:rep_fBm}, 
replacing~$\beta$ by~$\beta \eps^2$, and for $s,t\in\Tt$ with $0<s<t$, its pointwise limit
\begin{equation*}
G^H_0 (t,s) := \lim_{\eps \downarrow 0} G^H_\eps (t,s) = 
\left\{
\begin{array}{ll}
\displaystyle \frac{\xi\kappa_H \Hm}{s^{\Hm}} \int_s^t \frac{u^{\Hm}}{(u-s)^{1-\Hm}} \D u, 
& \displaystyle \text{for } H \in \left(\frac{1}{2}, 1\right),\\
\displaystyle \frac{\xi\kappa_H}{s^{\Hm}}\left({(t(t-s))}^{\Hm} 
- \Hm \int_{s}^{t} {(u-s)}^{\Hm} u^{\Hm - 1} \D u \right), & \displaystyle \text{for } H \in \left(0, \frac{1}{2}\right), \\
\xi, & \displaystyle \text{for }H=\frac{1}{2}.
\end{array}
\right.
\end{equation*}

The following proposition is similar to Proposition~\ref{prop:RKHS_tails}, 
as $G^H_0(t,\cdot) \in \LTt$ and for all $0<s<t$, $G^H_0(t,s) >0$, and its proof is omitted.
\begin{proposition}\label{Pp:RKHS_ST}
For any $H \in (0,1)$, the space $\Hh_{G^H_0}$ 
is the RKHS of the Gaussian process 
$\int_0^{\cdot} G^H_0 (\cdot,s) \D Z_s$.
\end{proposition}

%%%%%%%%%%%%%%%%%%%%%%%%%%%%%%%%%%%%%%%%%%%%%%%%%%%
%%%%%%%%%%%%%%%%%%%%%%%%%%%%%%%%%%%%%%%%%%%%%%%%%%%
\subsubsection{Small-time asymptotics for bounded support in the diffusion case}\label{Sec:Mellouk}

In the standard case $H=\frac{1}{2}$, the rescaling in the previous subsection is not really `natural', 
in the sense that small-time weak convergence usually holds for It\^o diffusions
without space rescaling.
In this case, using an approach introduced by Bezuidenhout~\cite{B87} and further developed 
by Mellouk~\cite{M00},
we can obtain simpler large deviations estimates if the support of the initial datum~$\Theta$ is bounded. 
A simplified version of Mellouk considers, for any $\eps>0$, the system, on~$\Tt$,
\begin{equation}\label{eq_ree}
\D \X^\eps_t = \bk(\X^\eps_t, \Z)\D t + \eps \ak(\X^\eps_t,\Z)\D \W_t, 
\qquad \X^\eps_0 = \xx_{0} \in \RR^n,
\end{equation}
where $\bk:\RR^n \times \RR^m \rightarrow \RR^n$ and 
$\ak: \RR^n \times \RR^m \rightarrow \mathcal{M}_{(n,d)}$ are bounded Borel measurable, 
uniformly Lipschitz continuous, 
$\Z$ is a random variable with compact support on~$\RR^m$ 
and~$\W$ a $d$-dimensional standard Brownian motion, independent of~$\Z$.
The main result of the paper is a large deviations principle on $\Cc^\alpha(\Tt,\RR^n)$, 
the space of $\alpha$-H\"older continuous functions,
for $0\le\alpha< \frac{1}{2}$, for the sequence $(\X^\eps)_{\eps >0}$,
under the following assumptions:
\begin{assumption}\ 
\label{assM}
\begin{itemize}
\item[(H1)] $\bk(\cdot,\cdot)$ is jointly measurable on $\RR^n \times \RR^m$ and there exists $C>0$ such that, for all $\xx,\xx' \in \RR^n$, $\zz,\zz' \in \RR^m$,
$$
|\bk(\xx,\zz)| \le C(1+|\xx|)
\qquad\text{and}\qquad
|\bk(\xx,\zz) - \bk(\xx',\zz')| \le C(|\xx-\xx|+|\zz-\zz'|).
$$
\item[(H2)] $\ak(\cdot,\cdot)$ is jointly measurable on $\RR^n \times \RR^m$ and there exists  $C>0$ such that, for all $\xx,\xx' \in \RR^n$, $\zz,\zz' \in \RR^m$,
$$
\|\ak(\xx,\zz)\| \le C
\qquad\text{and}\qquad
\|\ak(\xx,\zz) - \ak(\xx',\zz')\| \le C (|\xx-\xx'|+|\zz-\zz'|).
$$
\end{itemize}
\end{assumption}

For $f \in \Hh$, $u \in \supp(\Z)$ and $\xx_0 \in \RR^n$,
let $\Ss_{\xx_0}(f,u)$ denote the unique solution to the controlled ODE
$g_t = \xx_0 + \int_{0}^{t} \bk(g_s,u_s)\D s + \int_{0}^{t} \ak(g_s,u_s) \dot{f}_s \D s$, 
for $t \in \Tt$. 
Let us now introduce the following definition:

\begin{definition}\label{def:LSC}
Let $\alpha\in [0,\frac{1}{2})$ and~$\Bb_a$ be the ball of radius~$a$ in the $\alpha$-H\"older norm.
The lower semi-continuous regularisation~$\breve{I}:\Cc^\alpha(\Tt,\RR^m)\to\Cc^\alpha(\Tt,\RR^m)$ 
of a functional~$I:\Cc^\alpha(\Tt,\RR^m)\to\Cc^\alpha(\Tt,\RR^m)$ is defined as 
$$
\breve{I}(\psi) := \lim_{a \downarrow 0} \inf\left\{I(\varphi): \varphi\in \Bb_{a}(\psi)\right\}.
$$
\end{definition}

\begin{theorem}[Theorem 2.1 in~\cite{M00}]\label{Th_Mell}
Under Assumption~\ref{assM},
$(\X^\eps)_{\eps >0}\sim\LDP(\eps^2, \breve{I}^\alpha)$, 
where (with $\Lambda$ in~\eqref{eq:LDPVolterrageneral})
$$
I^\alpha(\psi) := \inf \left\{\Lambda(\varphi): \varphi \in \Hh, \Ss_{\xx_0}(\varphi,u) = \psi,
\text{ for some }u \in \supp(\Z) \right\}.
$$
\end{theorem}

Coming back to our model, the rescaling
$(X^{\eps}_t, Y^{\eps}_t) := (X_{\eps^2 t}, Y_{\eps^2 t})$,
equivalent to that of Section~\ref{sec:GenCase} with $b=0$ and $H=\frac{1}{2}$, 
the small-noise system~\eqref{eq:model}, under Assumption~\textbf{$\Aa$}, becomes,
similarly to Section~\ref{sec:GenCase},
\begin{equation}\label{eq:SystemEps}
\left\{
\begin{array}{ll}
\displaystyle
\D X^{\eps}_t = - \frac{\eps^{2}}{2} \widetilde{\sigma}(Y^{\eps}_t)^2 \D t + \eps \widetilde{\sigma}(Y^{\eps}_t) \left(\rho\D B_t + \overline{\rho} \D B^{\perp}_t\right), &X^{\eps}_0=0, \\
\D Y^{\eps}_t = \left(\eps^{2} \lambda + \beta \eps^2 Y^{\eps}_t \right)\D t + \eps \xi \D B_t, &Y^{\eps}_0 \sim \Theta,
\end{array}
\right.
\end{equation}
with~$B$ a standard Brownian motion.
Subtracting the initial random datum~$\X^\eps_0 = \X_0 = (0,\Theta)'$, this system can be expressed in the form~\eqref{eq_ree} with $\X^\eps = (X^\eps, Y^\eps)$, 
\begin{equation}\label{eq:SystemEpsCoef}
\bk(\eps,\X^\eps,\X_0) = 
\eps^{2}
\begin{pmatrix}
\displaystyle - \frac{1}{2} \widetilde{\sigma}(Y^{\eps}_t + \Theta)^2\\
\displaystyle \lambda + \beta (Y^{\eps}_t + \Theta)
\end{pmatrix}
\qquad\text{and}\qquad
a(\X^\eps,\X_0) = 
\begin{pmatrix}
\rho \widetilde{\sigma}(Y^{\eps}_t+\Theta) & \overline{\rho}\widetilde{\sigma}(Y^{\eps}_t+\Theta)\\
\xi & 0
\end{pmatrix},
\end{equation}
and note that~$\bk(\eps,\cdot,\cdot)$ converges to the null map as~$\eps$ tends to zero.
The assumptions imposed in~\cite{M00} on the drift and diffusion coefficients are clearly satisfied here. 
While Mellouk allows the drift and diffusion to depend explicitly on external random factors, 
we can write our setting (dependence on a random starting point) into this framework. 
The large deviations estimate for the sequence $(\X^\eps)_{\eps>0} = (X^\eps, Y^\eps)_{\eps >0}$ 
thus obtained is stronger than that in the previous section, as it holds on $\Cc^\alpha(\Tt,\RR^n)$, 
for any $0\le\alpha< \frac{1}{2}$. 
Note that the mild conditions on the coefficients~\cite[($H_0$)-($H_2$)]{M00} are easily satisfied in our case, so the only additional assumption is the boundedness of the support on~$\Theta$.
We remark here that~\cite{M00} is not directly applicable to the current setting 
but has to be extended to include $\eps$-dependence in the drift, and we do so following 
the Azencott~\cite{A80}'s inspired approach developed by Peithmann~\cite[Subsection 2.2.1]{P07}. 
In order to state LDP (proved in Section~\ref{sec:Proof-Th_rdm_new}), we impose the following assumption:
\begin{assumption}\label{ass:UniqueSol}
For $u \in \supp(\X_0)$ and $\varphi\in \Hh$, 
the ODE $\psi_t = \int_{0}^{t} \ak(\psi_s,u_s) \dot{\varphi}_s \D s$ 
has a unique solution on~$\Tt$, denoted by $\Ss_{0}(\varphi,u)$.
\end{assumption}

\begin{theorem}\label{Th_rdm_new}
Under Assumption~\ref{ass:UniqueSol}, if~$\Theta$ has compact support, then
$(\X^\eps)_{\eps >0}\sim\LDP(\eps^2,\breve{I}^\alpha)$, with
$$
I^\alpha(\psi) := \inf \left\{ \Lambda^{\rho}(\varphi): \varphi\in \Hh,
\text{ such that } \Ss_{0}(\varphi,u) = \psi, \text{ for some }u \in \supp(\X_0) \right\},
$$
with $\Lambda^{\rho}$ defined in~\eqref{eq:ext_grf}. 
\end{theorem}

%%%%%%%%%%%%%%%%%%%%%%%%%%%%%%%%%%%%%%%%%%%%%%%%%%%
%%%%%%%%%%%%%%%%%%%%%%%%%%%%%%%%%%%%%%%%%%%%%%%%%%%
\section{Applications to Implied volatility asymptotics}\label{sec:applications}
As announced in the introduction, we unify here two branches of research, 
both aimed at reproducing the steepness of the implied volatility surface on the short end via models with continuous paths.
While there are now numerous outputs~\cite{ALV07,BFG16,BFGHS17,FZ16,Fukasawa,GJR14,JPS17} 
in the literature confirming that a fractional driving noise (with Hurst exponent $H<1/2$) 
in the volatility leads to the observed steepness of the smile, 
recent results~\cite{JR15,JS17} reproduce this effect by randomising the initial volatility in classical diffusive models.
In this section we demonstrate how to modulate the two effects with respect to one another. 
In the Black-Scholes-Merton model, the price of a European Call option is $\Cc_{BS}(t,\E^k, \Sigma)$, 
with associated volatility~$\Sigma$.
Considering a market with observed Call option prices $\Cc_{obs}(t, \E^k)$, 
with maturity~$t$ and strike~$\E^k$, we denote by $\Sigma_{t}(k)$ the implied volatility,
defined as the unique non-negative solution to $\Cc_{BS}(t, \E^k, \Sigma_{t}(k)) = \Cc_{obs}(t, \E^k)$.

%%%%%%%%%%%%%%%%%%%%%%%%%%%%%%%
\subsection{General fractional case}
From Theorems~\ref{Th:tails_frac_X} and~\ref{Th:LDP_ST_X}, 
we can deduce the asymptotic behaviour of the implied volatility 
for large strikes and for small maturities.
We state those below, and postpone the proofs to Appendices~\ref{sec:Proof_IVTails} and~\ref{sec:Proof_IVST}.

\begin{corollary}(Large-strike implied volatility asymptotics)\label{Cor:ImpliedvolTails}
For any $H \in (0,1)$ and any $b \geq 1/2$ such that Theorem~\ref{Th:tails_frac_X}
holds, we have the following large-strike asymptotic estimates of the implied volatility:
$$
\lim_{k\uparrow \infty} \frac{\Sigma^2_t(k)t}{k}
= \frac{1}{2}\left(\inf_{y \ge 1} \widetilde{\Lambda}(\phi)|_{\phi_t =y}\right)^{-1}
\quad \textrm{with } \widetilde{\Lambda} \textrm{ as in~\eqref{eq:grf_X_tails}}, \text{ and for any } t \in \Tt.
$$
\end{corollary}

Similarly, from Theorem~\ref{Th:LDP_ST_X} we can deduce the asymptotic behaviour of the implied volatility when time becomes small. 
The following Corollary generalises~\cite[Corollary 4.10]{FZ16}.
\begin{corollary}[Small-time Implied volatility asymptotics]\label{Cor:ImpliedvolST}
For any $H \in (0,1)$ and any $b \geq 1/2-2H$ such that Theorem~\ref{Th:LDP_ST_X} holds, 
the following small-time estimate is true for any $k\ne 0$:
\begin{equation}\label{eq:IV_ST}
\lim_{t\downarrow 0} t^b \Sigma^2_t\left(t^{1/2-H-b}k\right)
 = \frac{k^2}{2}\left(\inf_{y \ge k} \II(\phi)|_{\phi_{1} =y}\right)^{-1}, 
 \quad \textrm{with } \II \textrm{ as in~\eqref{eq:grf_X_ST}}.
\end{equation}
\end{corollary}

This implies that the implied volatility explodes with rate~$t^{-b}$. 
For $b=0$, it is identical to~\cite[Formula (26)]{FZ16}.

%%%%%%%%%%%%%%%%%%%%%%%%%%%%%%%%%%%%
\subsection{Refined asymptotic results in the special diffusive case from Sections~\ref{app:MNS} and~\ref{Sec:Mellouk}}
\label{Ch3_rdm_future}
\subsubsection{Large-strike asymptotics}
We consider here a specific case of a multidimensional diffusion, as we are only interested 
in studying its tail asymptotics.
Let~$\mathrm{X}^{\zeta}:=(X^{\zeta, 1},\cdots, X^{\zeta,n})$ be the unique strong solution in~$\RR^n$ to 
$$
\D \mathrm{X}_{t}^{\zeta} = \widetilde{\bk}(\mathrm{X}_{t}^{\zeta})\D t + \ak(\mathrm{X}_{t}^{\zeta}) \D W_t,
\qquad \mathrm{X}_{0}^{\zeta} = \zeta,
$$
for some $d$-dimensional standard Brownian motion and some square integrable random variable~$\zeta$,
and with
$\widetilde{\bk}:\RR^n\to\RR^n$ and $\ak: \RR^{n} \rightarrow \mathcal{M}_{(n,d)}(\RR)$.
Consider the following scaling assumption:
\begin{assumption}\label{assu:Tails}
There exist $b_1,\cdots,b_n>0$ with $b_1=2$ such that 
$\X^{\eps,\zeta}:=\left(\eps^{b_1}X^{\zeta,1},\cdots,\eps^{b_n}X^{\zeta,n}\right)$
satisfies 
\begin{equation}\label{eq:ScalingTails}
\D \X_{t}^{\eps,\zeta} = \eps \widetilde{\bk}(\eps,\X_{t}^{\eps,\zeta})\D t
 + \eps \ak(\X_{t}^{\eps,\zeta}) \D W_t,
\qquad \X_{0}^{\zeta} = \left(\eps^{b_1}\zeta^{(1)},\cdots,\eps^{b_n}\zeta^{(n)}\right).
\end{equation}
Furthermore, 
$\eps \widetilde{\bk}(\eps, \cdot/\eps)$ converges uniformly to some function~$\bk(\cdot)$ as~$\eps$ tends to zero.
\end{assumption}
We can then state the main result about tail asymptotics:
\begin{proposition}\label{prop:RSP}
Under Assumption~\ref{assu:Tails}, if there exists~$\xx_0 \in \RR^n$ such that for any $\delta >0$, 
$\limsup_{\eps \downarrow 0} \eps^2 \log \PP(|\zeta-\xx_0|>\delta)=-\infty$, 
and if the triple $(\xx_0, \bk(\cdot), \ak(\cdot))$ satisfies Assumption~\ref{ass:MNS}, then 
$$
\lim_{\eps \downarrow 0} \eps^2 \log \PP(\eps X^{\zeta,1}_t \ge 1)
 = \lim_{\eps \downarrow 0} \eps^2 \log \PP(\eps X^{\xx_0,1}_t \ge 1),
 \qquad\text{for any }t\in\Tt.
$$
\end{proposition}
The scaling from Assumption~\ref{assu:Tails} may be odd at first, but reflects the fact that 
components of stochastic models may each act on different scales. 
Consider for example the Ornstein-Uhlenbeck process, solution to
\begin{equation*}
\begin{array}{rll}
\D X_t & = -\frac{1}{2}Y_t^2 \D t + Y_t \D W_t, & X_0 = \zeta,\\
\D Y_t & = (\lambda + \beta Y_t) \D t + \xi \D B_t, & Y_0 = y_0>0,
\end{array}
\end{equation*}
where $W$ and $B$ are two correlated Brownian motions.
The rescaling $(X^\eps,Y^\eps) := (\eps^2 X, \eps Y)$ 
(corresponding to $b_1=2$ and $b_2=1$) yields
\begin{equation*}
\begin{array}{rll}
\D X_t^{\eps} & = -\frac{1}{2}(Y_t^{\eps})^2 \D t + \eps Y_t^{\eps} \D W_t, & X_0^{\eps} = \eps^2\zeta,\\
\D Y_t^{\eps} & = (\eps\lambda + \beta Y_t^{\eps}) \D t + \eps\xi \D B_t, & Y_0^{\eps} = \eps y_0>0,
\end{array}
\end{equation*}
namely~\eqref{eq:ScalingTails}, and the assumptions are satisfied.

The proof of Proposition~\ref{prop:RSP} is postponed to Appendix~\ref{sec:Proof_prop_RSP}.
The assumption of the random initial condition $ \X^\zeta_0 = \zeta$ being $\mathcal{F}_0$-measurable distribution can be relaxed.
Indeed, $\mathcal{F}$ is the filtration generated by the $d$-dimensional Brownian motion~$\W$. 
Instead, one could work with a filtration $\Ff' := {(\Ff'_t)}_{t\in\Tt^*}$ generated by 
$\mathcal{F}'_t := \sigma(\left\{\W_u, u\le t\right\} \cup \left\{\zeta\right\})$, for all $t\in\Tt^*$. 
Then the random initial point~$X^\zeta_0$ has a $\Ff'_0$-measurable distribution 
and the results above still hold, in particular Theorem~\ref{thm:MNSLDP}  and Proposition~\ref{prop:RSP}, 
on the new filtered probability space $(\Omega, \Ff', {(\Ff')}_{t\in\Tt}, \PP)$.
In the context of implied volatility asymptotics, this result has the following meaning:
\begin{corollary}
\label{AsympIV}
The wings of the smile are independent of the starting point ($\zeta$ or $\xx_0$).
\end{corollary}
\begin{proof}
Gao and Lee~\cite{GL14} show that asymptotic behaviour of the implied volatility can be directly inferred 
from comparing tail probabilities to those of the Black-Scholes model.
It is straightforward to see that the scaling of Proposition~\ref{prop:RSP} is the same in Black-Scholes,
and the corollary follows immediately.
\end{proof}

%%%%%%%%%%%%%%%%%%%%%%%%%%%%
\subsubsection{Small-time asymptotics for the `forward' Stein-Stein model}
We are interested in a `forward' process, as defined by Jacquier and Roome~\cite{JR15}
in the context of forward-start European options:
$$
\mathbb{E} {\left( \frac{S_{t+\tau}}{S_t} - \E^k\right)}^{+} 
 = \mathbb{E} {\left( \E^{X_{t+\tau}-X_t} - \E^k\right)}^{+}
 =: \mathbb{E} {\left(\E^{X_{\tau}^{(t)}} - \E^k\right)}^{+},
$$
with ${(X_\tau^{(t)})}_{\tau \ge 0}$ the so-called `forward' process, defined path-wise by 
$X_\tau^{(t)} := X_{t+\tau} - X_t$, for some fixed $t>0$, and for all $\tau\geq 0$.
The 'forward' process ${(X_\tau^{(t)})}_{\tau \ge 0}$ then satisfies the following stochastic differential equation:
\begin{equation}\label{eq:rSS}
\left\{
\begin{array}{ll}
\displaystyle \D X_\tau^{(t)}  = -\frac{1}{2} (Y^{(t)}_\tau)^2 \D \tau + Y^{(t)}_\tau \D W_{1,\tau}, & X_0^{(t)} = 0,\\
\displaystyle \D Y^{(t)}_\tau = (\lambda + \beta Y^{(t)}_\tau)\D \tau + \xi \D W_{2,\tau}, & Y^{(t)}_0 ~\sim \sigma_t.
\end{array}
\right.
\end{equation}
The stochastic differential equation for $(X_\tau^{(t)} , Y_\tau^{(t)})_{\tau\geq 0}$ 
is the same as that for $(X_t , Y_t)_{t\geq 0}$, 
albeit with an initial random distribution $(\delta_0,\sigma_t)$, 
where~$\sigma_t$ is Gaussian 
with mean $\E^{\beta t}(\sigma_0 + \frac{\lambda}{\beta}) - \frac{\lambda}{\beta}$ 
and variance $\frac{\xi^2}{2\beta}(\E^{2\beta t}-1)$. 
We now apply the results of Section~\ref{Sec:Mellouk} to obtain  small-time asymptotics for a version of the Stein-Stein `forward' model 
with a generalised random starting point. 
\begin{proposition}
With the scaling $(X^\eps_\tau, Y^\eps_\tau) := (X^{(t)}_{\eps^2 \tau}, Y^{(t)}_{\eps^2 \tau})$ for $\eps,t >0$, the randomised Stein-Stein rescaled model~\eqref{eq:rSS} 
is the same as~\eqref{eq:SystemEps} with coefficients given in~\eqref{eq:SystemEpsCoef},
with $\widetilde{\sigma}(y)\equiv y$, and Theorem~\ref{Th_rdm_new} applies.
\end{proposition}
We can translate this result into forward implied volatility asymptotics directly using~\cite[Theorem 2.4]{JF09},
and refer to this very paper for a precise definition of the forward implied volatility~$\Sigma_{t,\tau}$:
\begin{corollary}
The small-time forward smile reads
$\displaystyle \lim_{\tau \downarrow 0} \Sigma^2_{t,\tau}(k)
 = \frac{k^2}{2}\left(\inf_{y \ge k} \breve{I}^\alpha (\phi)|_{\phi_{1}=y}\right)^{-1}$, 
 with~$\breve{I}^\alpha$ in Theorem~\ref{Th_rdm_new}.
\end{corollary}
%%%%%%%%%%%%%%%%%%%%%%%%%%%%%%%%%%%%%%%%%%%%%%%%
%%%%%%%%%%%%%%%%%%%%%%%%%%%%%%%%%%%%%%%%%%%%%%%%

\appendix
%%%%%%%%%%%%%%%%%%%%%%%%%%%%%%%%%%%%%%%%%%%%%%%%
%%%%%%%%%%%%%%%%%%%%%%%%%%%%%%%%%%%%%%%%%%%%%%%%
\section{Proofs}

\subsection{Proof of Theorem~\ref{Th:tails_frac_X}}
\subsubsection{Proof of Lemma~\ref{lem:rep_fBm}}\label{App:Lem_F}
From~\cite[Definition 2.1]{YLX08}, for any smooth function~$f$ on $\Tt$ with bounded derivatives such that 
$\|f\|:= \int_{\Tt} {(\Gamma^{*}_{H,t} f(t))}^2 \D t < \infty$, define, for $t\in\Tt$,
$ \int_{0}^{t} f(u) \D W^H_u := \int_{0}^{t} \Gamma^{*}_{H,t} f(u) \D Z_u $, with 
$$
\Gamma^{*}_{H,t} f(s) := - \frac{\kappa_H}{s^{\Hm}} \frac{\D}{\D s} \int_{s}^{t} u^{\Hm} {(u-s)}^{\Hm} f(u) \D u,
\qquad\text{for all }s\in (0,t].
$$
Applying this to $f(s) := \xi\E^{\beta (t-s)}$, one obtains the following:

\begin{itemize}
\item[(1)] for $H < \frac{1}{2}$, using integration by part and Leibniz' integration rule,
\begin{align*}
 & \frac{\D}{\D s} \int_{s}^{t} u^{\Hm} {(u-s)}^{\Hm} f(u) \D u \\
&= \frac{\D}{\D s} \left\{ \frac{\xi(t-s)^{\Hp} t^{\Hm}}{\Hp} \right\} +
 \frac{\D}{\D s} \left\{-\xi\frac{\Hm}{\Hp} \int_{s}^{t} \frac{(u-s)^{\Hp}}{u^{1-\Hm}} \E^{\beta (t-u)} \D u + \xi\frac{\beta}{\Hp} \int_{s}^{t} {(u-s)}^{\Hp} u^{\Hm} \E^{\beta (t-u)} \D u\right\} \\
& =- \xi{(t-s)}^{\Hm} t^{\Hm} + \xi\Hm \int_{s}^{t} \frac{(u-s)^{\Hm}}{u^{1-\Hm}} \E^{\beta (t-u)} \D u 
- \xi\beta \int_{s}^{t} {(u-s)}^{\Hm} u^{\Hm} \E^{\beta (t-u)} \D u.
\end{align*} 
Hence,
\begin{align*}
\Gamma^{*}_{H,t} f(s) 
&:= -\frac{\kappa_H}{s^{\Hm}} \frac{\D}{\D s} \int_{s}^{t} u^{\Hm} {(u-s)}^{\Hm} f(u) \D u \\
&= \frac{\xi\kappa_H}{s^{\Hm}} \left((t-s)^{\Hm} t^{\Hm} 
+\beta \int_{s}^{t} {(u-s)}^{\Hm} u^{\Hm} \E^{\beta (t-u)} \D u
 - \Hm \int_{s}^{t} \frac{(u-s)^{\Hm}}{u^{1-\Hm}} \E^{\beta (t-u)} \D u \right).
\end{align*}

\item[(2)] for $H > \frac{1}{2}$, using Leibniz' integration rule,

$$
\frac{\D}{\D s} \int_{s}^{t} u^{\Hm} {(u-s)}^{\Hm} f(u) \D u
= \int_s^t \xi\frac{\D}{\D s} u^{\Hm} {(u-s)}^{\Hm} \E^{\beta (t-u)} \D u
= -\xi\Hm \int_s^t u^\Hm (u-s)^{\Hm-1} \E^{\beta (t-u)} \D u,
$$
and hence
$$
\Gamma^{*}_{H,t} f(s) 
:= -\frac{\kappa_H}{s^{\Hm}} \frac{\D}{\D s} \int_{s}^{t} u^{\Hm} {(u-s)}^{\Hm} f(u) \D u
=\frac{\xi\kappa_H \Hm}{s^{\Hm}} \int_s^t u^\Hm (u-s)^{\Hm-1} \E^{\beta (t-u)} \D u
$$
\end{itemize}

%%%%%%%%%%%%%%%%%%%%%%%%%%%%%%%%
\subsubsection{Proof of Proposition~\ref{prop:RKHS_tails}}\label{sec:proof_RKHS_tails}
In order to prove that~$\Hh_{\Fh}$ is the RKHS of the Gaussian process $\xi\int_0^\cdot \E^{\beta (\cdot -s)}\D W^H_s$, one needs to show that the operator~$\KK_{\Fh}$ is a bijection from~$\LTt$ to~$\Hh_{\Fh}$ and that $\Hh_{\Fh}$ is dense in $\mathcal{C}$.

The operator $\KK_{\Fh}$ acting on~$\LTt$ 
defined by $(\KK_{\Fh}f)(t):=\int_{0}^{t} \Fh(t,s) f(s) \D s$
satisfies $\KK_{\Fh} \LTt =: \Hh_{\Fh}$ and is surjective. 
It is also injective on~$\Hh_{\Fh}$: let $f \in \LTt$ a non-zero function such that  
$(\KK_{\Fh}f) = 0$ on~$\Tt^*$. 
Similar to~\cite{FZ16}, there is an interval $[t_1,t_2]\subset \Tt$ where~$f$ has constant sign. 
Using previous notations, $\Fh$ is defined, for $t\in\Tt^*$ and $s\in (0,t]$, as 
$\Fh(t,s) := \xi\Gamma^{*}_{H,t}\E^{\beta (t-s)}$. 
The function $g(s) := \xi\int_{t}^{s} u^{\Hm} {(u-s)}^{\Hm} \E^{\beta (t-u)} \D u$ is increasing, 
hence~$\Fh(t,\cdot)$ is a positive function on $[0,t]$ and
$\int_{t_1}^{t_2} \Fh(t,s) f(s) \D s >0$, leading to a contradiction. 
Thus $\KK_{\Fh}$ is injective on $\Hh_{\Fh}$, hence bijective from~$\LTt$ to~$\Hh_{\Fh}$.	
Since $\KK_{\Fh}$ is a linear operator, 
one can then define 
${\langle \KK_{\Fh}f_1, \KK_{\Fh}f_2 \rangle}_{\Hh_{\Fh}} := {\langle f_1, f_2 \rangle}_{\LTt}$ as an inner product.

The second part of the proof consists in showing that $\Hh_{\Fh} :=\KK_{\Fh}\LTt$ is dense in~$\Cc$. 
We claim that $\Fh(t,\cdot)^{-1} \in \LTt$.
Indeed, as shown in the proof of~\cite[Theorem 3.1]{YLX08}, the following hold for all $0<s<t$:
\begin{itemize} 
\item if $0<H<\frac{1}{2}$, $\Fh(t,s) \ge \xi\kappa_H s^{\frac{1}{2}-H} \tau^{\Hm} {(t-s)}^{\Hm} >0$, leading
$$
\int_{0}^{t} \frac{\D s}{ \Fh(t,s)^2}
\le \frac{1}{\xi^2 \kappa^2_H t^{2\Hm}} \int_{0}^{t} \left(\frac{s}{t-s}\right)^{2\Hm}\D s
\le \frac{1}{\xi^2 \kappa^2_H t^{4\Hm}} \int_{0}^{t} s^{2\Hm} \D s
= \frac{{t^{2-2H}}}{2H\xi^2 \kappa^2_H} < \infty;
$$
\item if $\frac{1}{2} < H < 1$, $\Fh(t,s) \ge \xi \kappa_H \E^{- \beta( s-t)} (t-s)^{\Hm}$, so that
$$
\int_{0}^{t} \frac{\D s}{\Fh(t,s)^2}
\le \frac{1}{\xi^2 \kappa^2_H \E^{2\beta t}} \int_0^t \frac{\E^{2\beta s}\D s}{(t-s)^{2\Hm}}
\le \frac{1}{\xi^2 \kappa^2_H \E^{2\beta t} t^{2\Hm}} \int_0^t \E^{2\beta s} \D s
= \frac{1-\E^{-2\beta t}}{2\beta\xi^2 \kappa^2_H t^{2\Hm}} 
< \infty.
$$
\end{itemize}
Given $a>0$, take for $s \in [0,t]$, $f_a(s) := \frac{s^a}{\Fh(t,s)} \in \LTt$. 
Then, $\int_{0}^{t} \Fh(t,s) f_a(s) \D s = \frac{{t}^{a+1}}{a+1}$.
Hence, $\Hh_{\Fh}$ contains all polynomials null at the origin, and by the Stone-Weierstrass Theorem, 
is dense in~$\Cc$.

%%%%%%%%%%%%%%%%%%%%%%%%%%%%%%%%%%%%%%%%%%%%%%%%
%%%%%%%%%%%%%%%%%%%%%%%%%%%%%%%%%%%%%%%%%%%%%%%%
\subsubsection{Proof of Theorem~\ref{Th:tails_frac_X}(i)}\label{Sec:Proof1}
Let $\overline{Y}^{\eps}$ denote the solution to the second SDE~\eqref{eq:fracSSRandom}, starting from zero.
The product rule for (fractional) Brownian motion~\cite[Theorem 3.1.4]{M08} yields, for any $t \geq 0$,
\begin{equation}\label{eq:auxi}
\overline{Y}^{\eps}_t = - \frac{\lambda \eps^b}{\beta} \left(1- \E^{\beta t}\right) + \xi \eps^b \int_{0}^{t} \E^{\beta (t-u)} \D W^H_u,
\end{equation}
Since 
$\xi \int_{0}^{\cdot} \E^{\beta (\cdot-u)} \D W^H_u$ is Gaussian,
Lemma~\ref{lem:rep_fBm} combined with~\cite[Theorem 3.4.5]{DBookS}
yields a large deviations principle on~$\Cc$ for the sequence
$\left(\xi\eps^b \int_{0}^{\cdot} \E^{\beta (\cdot - s)} \D W^H_s\right)_{\eps>0}$,
with speed~$\eps^{2b}$ and good rate function~$\Lambda_{\Fh}$ as in~\eqref{eq:LDPVolterrageneral}.
Since the two sequences ${(\overline{Y}^{\eps})}_{\eps >0}$ and 
${(\xi \eps^b \int_{0}^{\cdot} \E^{\beta (\cdot - s)} \D W^H_s)}_{\eps >0}$ only differ by some deterministic quantity, they are clearly exponentially equivalent:
for any $\delta >0$ and $t\in\Tt^*$, 
$$
\limsup_{\eps \downarrow 0} \eps^{2b}
\log \PP \left( \left\|\overline{Y}^{\eps}-\xi \eps^b \int_{0}^{\cdot} \E^{\beta (\cdot-s)} \D W^H_s\right\|_{\infty} > \delta \right)
= \limsup_{\eps \downarrow 0} \eps^{2b} \log \PP \left( \left\|\frac{\lambda \eps^b}{\beta} (\E^{\beta \cdot}-1)\right\|_{\infty}> \delta \right) = -\infty,
$$
since $\beta <0$,
so that $\overline{Y}^{\eps}\sim\LDP\left(\eps^{2b}, \Lambda_{\Fh}\right)$.
Finally, the LDP for~$Y^{\eps}$ follows again by exponentially equivalence:  
$$
\PP\left(\left\|Y^{\eps}-\overline{Y}^{\eps} \right\|_{\infty}> \delta\right)
= \PP\left( \left\| \eps^b \Theta \E^{\beta \cdot}\right\|_{\infty} > \delta\right)
\le \PP \left( \left|\Theta\right| > \frac{\delta}{\eps^b } \right),
$$
for $t,\eps,\delta >0$, and the theorem follows from Assumption~$\Abt$.

%%%%%%%%%%%%%%%%%%%%%%%%%%%%%%%%%%%%%%%%%%%%%%%%
%%%%%%%%%%%%%%%%%%%%%%%%%%%%%%%%%%%%%%%%%%%%%%%%
\subsubsection{Proof of Theorem~\ref{Th:tails_frac_X}(ii)}
We first prove large deviations for the Gaussian drivers of the process,
which we then, by means of iterated Contraction Principles, translate to large deviations
for the whole scaled process.
\begin{itemize}
\item[(1)] When $H \neq \frac{1}{2}$, Lemma B.1 in~\cite{FZ16} implies that $(\rho B,W^H)$ is Gaussian, 
so that $\eps^\bk(\rho B,W^H)$ satisfies a large deviations principle on $\Cc(\Tt,\RR^2)$, 
with speed $\eps^{2b}$ and good rate function
\begin{equation}\label{grf_LDP_X_1}
I_1(\varphi, \psi) :=
\left\{
\begin{array}{lccc}
\displaystyle
\frac{1}{2} \int_{\Tt} f(s)^2 \D s, & (\varphi, \psi) \in \overline{\Hh}_H, \\
+ \infty, & (\varphi, \psi) \notin \overline{\Hh}_H,
\end{array}
\right.
\end{equation}
where $\overline{\Hh}_H$ is the RKHS for $(\rho B,W^H)$ defined as
$$
\overline{\Hh}_H := \left\{(\varphi, \psi) \in \Cc^2: \varphi(\cdot) = \rho\int_{0}^{\cdot} f(s) \D s
\text{ and } \psi(\cdot)=\int_{0}^{\cdot} K^H(\cdot,s) f(s) \D s, \text{ for some }f \in \LTt\right\}.
$$
By independence of $B^\perp$ with respect to $B$ and $W^H$ and using the Contraction Principle, 
$\eps^\bk(\overline{\rho}B^\perp + \rho B, W^H)$ satisfies a large deviations principle on $\Cc(\Tt,\RR^2)$, with speed $\eps^{2b}$ and good rate function
$$
I_3(\varphi, \psi) :=
\inf_{\varphi=u+w} \left\{I_1(u,\psi)+I_2(w)\right\}
 = \inf_{u\in \Hh} \left\{I_1(u,\psi) + I_2(\varphi-u)\right\}, 
$$
where, with $\Hh_{\overline{\rho}}$ as in~\eqref{eq:RKHSgeneral},
\begin{equation*}
I_2(w) :=
\left\{
\begin{array}{lccc}
\displaystyle
\frac{\overline{\rho}}{2} \int_0^{1} f(s)^2 \D s, & w \in \Hh_{\overline{\rho}}, \\
+ \infty, & w \notin \Hh_{\overline{\rho}}.
\end{array}
\right.
\end{equation*}

\item[(2)] For $H = \frac{1}{2}$, using the proof of Theorem~\ref{thm:ext_MNSLDP}, $\eps^b (\overline{\rho}B^\perp + \rho B,B)'$  
satisfies a large deviations principle on $\Cc(\Tt,\RR^2)$, when~$\eps$ tends to zero with speed $\eps^{2b}$ 
and good rate function $\Lambda^{\rho}$ defined in~\eqref{eq:ext_grf}.
\end{itemize}

We now introduce the process $\widetilde{X}^\eps$ satisfying the following SDE $\D \widetilde{X}^\eps_t = -\frac{1}{2} \widetilde{\sigma}^2(Y^\eps_t) \D t + \eps^b \widetilde{\sigma} (Y^\eps_t) (\rho \D B_t + \overline{\rho} \D B^\perp_t)$,
and translate this large deviations into one for the sequence~$\widetilde{X}^\eps$.
Since the proof for the $H \neq \frac{1}{2}$ case is the same as the one for $H = \frac{1}{2}$,
albeit with a different rate function, we concentrate on the former case.
Since $Y^{\eps}_t = \eps^b \E^{\beta t} \Theta - \frac{\lambda \eps^b}{\beta}(1-\E^{\beta t}) + \eps^b \xi \int_0^t \E^{\beta (t-u)} \D W^H_u$, 
one can define a continuous map~$G((0,1)\times\Cc\to\Cc)$, 
such that 
$(\eps^b (\rho B + \overline{\rho} B^{\perp}), Y^{\eps})=(\eps^b (\rho B + \overline{\rho} B^{\perp}), G(\eps,\eps^b W^H)(t))$. 
Moreover, using $\Abt$, $G(\eps,\eps^b W^H)$ is exponentially equivalent to $\overline{G}(\eps^b W^H)$, with $\overline{G}(\varphi)(t):= \xi \int_0^t \E^{\beta (t-u)} \D \varphi_u$.
The Contraction Principle thus yields
$(\eps^b (\rho B + \overline{\rho} B^{\perp}), Y^{\eps})\sim\LDP(\eps^{2b},I_4)$,
with $I_4:(\varphi,\psi)\mapsto\inf \{ I_3(\varphi,v) \ | \ \psi \in \Cc : \psi=\overline{G}(v) \}$.

Under~$\Aa'$, one can apply the extended Contraction Principle proved in~\cite[Proposition 2.3]{MNP92}, 
so that 
$(\eps^b (\rho B + \overline{\rho}) B^\perp, \widetilde{\sigma}(Y^{\eps}))\sim\LDP(\eps^{2b}, I_5)$ 
with 
$I_5(\varphi, \psi) := \inf \left\{ I_4(\varphi,v) \ | \ \psi=\widetilde{\sigma}(v) \right\}
 = \inf \left\{ I_3(\varphi,v) \ | \ \psi=\widetilde{\sigma}(\overline{G}(v)) \right\} $.
Finally, setting $b \ge \frac{1}{2}$, the sequence of semi-martingales $(\eps^b W)$ 
is uniformly exponentially tight (UET) in the sense of~\cite[Definition 1.1]{Gar08},
and the sequence $(\widetilde{\sigma}(Y^\eps))$ is c\`adl\`ag (Assumption $\Aa$), 
and adapted to the filtration $\mathcal{F}$. 
Denoting $X \cdot Y:=\int X\D Y$ the stochastic integral with respect to a semi-martingale, 
Theorem 1.2 in~\cite{Gar08} yields a large deviations principle on $\Cc(\Tt,\RR^3)$ for
$(\eps^b (\rho B_s + \overline{\rho} B^\perp_s), \widetilde{\sigma}(Y^\eps), \eps^b \int_0^\cdot \widetilde{\sigma}(Y^\eps_s) (\rho \D B_s + \overline{\rho} \D B^\perp_s))$, 
with speed $\eps^{2b}$ and good rate function~$I_6$ defined as
$I_6(\phi):= I_5(\varphi, \psi)$ if $\phi = \varphi \cdot \psi$ and $\psi\in\BV$
(the space of functions of finite variation), and infinity otherwise.
Applying another Contraction Principle, since, for $t\in\Tt^*$,
$$
\widetilde{X}^\eps_t = -\frac{1}{2} \int_0^t \widetilde{\sigma}(Y^\eps_s)^2 \D s + \int_0^t \eps^b \widetilde{\sigma}(Y^\eps_s) \D W_s
 =: I\Big(\widetilde{\sigma}(Y^\eps), 
 \widetilde{\sigma}(Y^\eps) \cdot \eps^b \left(\rho B_s + \overline{\rho} B^\perp_s\right)\Big)(\eps,t).
$$
Hence $\widetilde{X}^\eps\sim\LDP(\eps^{2b}, \widetilde{\Lambda})$ with
\begin{equation}\label{eq:grf_X_tails}
\widetilde{\Lambda} (\phi) := \inf \{ I_6(\chi) \ | \ \phi = I(\varphi, \varphi \cdot \psi), \chi = (\varphi,\psi), \psi\in\BV\}.
\end{equation}
The last step is proving that the processes $X^\eps$ and $\widetilde{X}^\eps$ are exponentially equivalent. Indeed, for $t\in\Tt^*$,
$$
\left\|X^\eps - \widetilde{X}^\eps \right\|_{\infty}
\le \frac{1}{2} \left\| \int_0^\cdot \left| \eps^{2b} \sigma^2 \left( \frac{Y^\eps_s}{\eps^b}\right) - \widetilde{\sigma}(Y^\eps_s) \right| \D s \right\|_{\infty} 
+ \eps^b \left\| \int_0^\cdot \left| \eps^b \sigma \left( \frac{Y^\eps_s}{\eps^b} \right) - \widetilde{\sigma} (Y^\eps_s) \right| \D \left(\rho B_s + \overline{\rho} B^\perp_s\right) \right\|_{\infty}.
$$
Using the linear growth assumption as well as Assumption~$\Aa$, we have for $s \in [0,t]$,
$$
\left| \eps^{2b} \sigma^2 \left( \frac{Y^\eps_s}{\eps^b}\right) - \widetilde{\sigma}(Y^\eps_s) \right| 
\le 2\eps^{2b} C^2 \left( 1 + \left|\frac{Y^\eps_s}{\eps^b}\right|^2 \right) 
+ 2C^2 \left( 1 + |Y^\eps_s|^2 \right)
\le 2C^2 \left( 1 + \eps^{2b} + 2|Y^\eps_s|^2 \right),
$$ 
hence
$$
\left\| \int_0^\cdot \left| \eps^{2b} \sigma^2 \left( \frac{Y^\eps_s}{\eps^b}\right) - \widetilde{\sigma}(Y^\eps_s) \right| \D s \right\|_{\infty} 
\le \left\| \int_0^\cdot 2C^2 \left( 1 + \eps^{2b} + 2|Y^\eps_s|^2 \right) \D s \right\|_{\infty} 
\le 2C^2 (1+ \eps^{2b}) + 4C^2 \left\| \int_0^\cdot |Y^\eps_s|^2 \D s \right\|_{\infty}. 
$$

Since $Y^\eps_s = \eps^b \E^{\beta s} \Theta + \eps^b \frac{\lambda}{\beta} (\E^{\beta s}-1) + \eps^b \xi \left[ W^H_s +\beta \int_0^s W^H_u \E^{-\beta u} \D u \right]$, 
we obtain for $s \in \Tt$, recalling that $\beta<0$,
$$
|Y^\eps_s|
\le \eps^b \E^{\beta s} |\Theta|
+ \eps^b \frac{\lambda}{\beta} \left(\E^{\beta s}-1\right) 
+ \xi \eps^b |W^H_s|,
$$
and
\begin{equation}\label{eq:intYeps2}
|Y^\eps_s|^2
\le 2\left(\eps^{2b} \E^{2\beta s} \Theta^2 
+ 2\eps^{2b} {\left(\frac{\lambda}{\beta}\right)}^2 \left(\E^{2\beta s}+1\right) 
+ \xi^2 \eps^{2b} |W^H_s|^2\right), 
\end{equation}
so that
$\int_0^t |Y^\eps_s|^2 \D s 
\le 2\left( \eps^{2b} t \Theta^2 +\eps^{2b} {\left(\frac{\lambda}{\beta}\right)}^2 \left(2t-\frac{1}{\beta}\right) + \xi^2 \eps^{2b} \int_0^t |W^H_s|^2 \D s \right)$.

Introducing $\Jj := \left\| \int_0^\cdot |Y^\eps_s|^2 \D s \right\|_\infty$, we obtain for any $\delta >0$ and $a_0 >0$,
$$
\begin{aligned}
\PP\left(\left\|X^\eps - \widetilde{X}^\eps \right\|_{\infty} > \delta \right)
&\le \PP\left(2C^2 \Jj + \eps^b \left\| \int_0^\cdot \left( \eps^b \sigma \left(\frac{Y^\eps_s}{\eps^b} \right) - \widetilde{\sigma} (Y^\eps_s) \right) \D (\rho B_s + \overline{\rho} B^\perp_s) \right\|_\infty > \delta_{\eps} \right)\\
&\le \PP\left(\left. 2C^2 \Jj + \eps^b \left\| \int_0^\cdot \left( \eps^b \sigma \left(\frac{Y^\eps_s}{\eps^b} \right) - \widetilde{\sigma} (Y^\eps_s) \right) \D (\rho B_s + \overline{\rho} B^\perp_s) \right\|_\infty > \delta_{\eps} \right| 2C^2 \Jj < a_0 \right)(1-\mathfrak{P})\\
& +  \PP\left(\left.2C^2 \Jj + \eps^b \left\| \int_0^\cdot \left( \eps^b \sigma \left(\frac{Y^\eps_s}{\eps^b} \right) - \widetilde{\sigma} (Y^\eps_s) \right) \D (\rho B_s + \overline{\rho} B^\perp_s) \right\|_\infty > \delta_{\eps} \right| 2C^2 \Jj \ge a_0 \right)\mathfrak{P}\\
&\le \PP\left(\left. \left\| \int_0^\cdot \left( \eps^b \sigma \left(\frac{Y^\eps_s}{\eps^b} \right) - \widetilde{\sigma} (Y^\eps_s) \right) \D (\rho B_s + \overline{\rho} B^\perp_s) \right\|_\infty 
> \frac{\delta_{\eps}-a_0}{\eps^b} \right| 2C^2 \Jj < a_0 \right)+\mathfrak{P},
\end{aligned}
$$
with $\mathfrak{P}:=\PP \left( 2C^2 \Jj \ge a_0 \right)$
and $\delta_{\eps}:=\delta - C^2(1+\eps^{2b})$.
Using~\eqref{eq:intYeps2}
and introducing $\Jj_t := \int_0^t |W^H_s|^2 \D s$, we obtain
$$
\begin{aligned}
\PP\left(\Jj > \bar{\delta} \right)
&\le \PP\left(\left| \eps^{2b} \Theta^2 +\eps^{2b} {\left(\frac{\lambda}{\beta}\right)}^2 \left(2-\frac{1}{\beta}\right) + \xi^2 \eps^{2b} \Jj_1 \right|  > \frac{\bar{\delta}}{2} \right)\\
&\le \PP\left(\left.\eps^{2b} \Theta^2 +\eps^{2b} {\left(\frac{\lambda}{\beta}\right)}^2 \left|2-\frac{1}{\beta}\right| + \xi^2 \eps^{2b} \Jj_1  > \frac{\bar{\delta}}{2} \right| \xi^2 \eps^{2b} \Jj_1 < a \right) \times \PP \left( \xi^2 \eps^{2b} \Jj_1 < a \right) \\
& +  \PP\left(\left.\eps^{2b} \Theta^2 +\eps^{2b} {\left(\frac{\lambda}{\beta}\right)}^2 \left|2-\frac{1}{\beta}\right| + \xi^2 \eps^{2b} \Jj_1  > \frac{\bar{\delta}}{2} \right| \xi^2 \eps^{2b} \Jj_1 \ge a \right) \times \PP \left( \xi^2 \eps^{2b} \Jj_1 \ge a \right)\\
&\le \PP\left(\left.\Theta^2 > \frac{1}{\eps^{2b}} \left( \frac{\bar{\delta}}{2} - \eps^{2b} {\left(\frac{\lambda}{\beta}\right)}^2 \left| 2-\frac{1}{\beta}\right|- a\right)\right| \xi^2 \eps^{2b} \Jj_1 < a \right) + \PP \left( \xi^2 \eps^{2b} \Jj_1 \ge a \right), \\
&\le \PP\left(\left.|\Theta| > \frac{1}{\eps^b} \sqrt{\left( \frac{\bar{\delta}}{2} - \eps^{2b} {\left(\frac{\lambda}{\beta}\right)}^2 \left| 2-\frac{1}{\beta}\right|- a\right)}\right| \xi^2 \eps^{2b} \Jj_1 < a \right) + \PP \left( \xi^2 \eps^{2b} \Jj_1 \ge a \right),
\end{aligned}
$$
for any $\bar{\delta} >0$ and $a >0$.
One can then use Markov's inequality: 
$$
\PP \left( \xi^2 \eps^{2b} \Jj_1 \ge a \right)
\le \frac{\mathbb{E}\left[ \xi^2 \eps^{2b} \Jj_1\right]}{a}
= \frac{\xi^2 \eps^{2b}}{a} \mathbb{E} \left(\int_0^1 |W^H_s|^2 \D s \right).
$$
Since $\mathbb{V}(W^H_s)=s^{2H}$, one obtains
$
\PP \left( \xi^2 \eps^{2b} \Jj_1 \ge a \right)
\le  \frac{\xi^2 \eps^{2b}}{a} \frac{1}{2H+1}.
$
Hence there exist $\eta \in [0,1)$ and $\bar{\eps}>0$ such that for all $\eps \le \bar{\eps}$, $\PP(\xi^2 \eps^{2b} \Jj_1 < a) \ge 1-\eta$.
Assumption~$\Abt$ then implies that there exist $M, \tilde{\eps} >0$ 
such that 
$$
\PP\left(|\Theta| > \frac{1}{\eps^b} \sqrt{\left( \frac{\bar{\delta}}{2} - \eps^{2b} {\left(\frac{\lambda}{\beta}\right)}^2 \left| 2-\frac{1}{\beta}\right|- a\right)} \right)
\le \exp\left(-\frac{M}{\eps^{2b}}\right),
\qquad\text{for all }\eps \le \tilde{\eps}.
$$
Thus, for all $\eps \le \min \{ \tilde{\eps}, \bar{\eps}\}$, Bayes' Theorem yields
$$
\begin{aligned}
&\PP\left(\left.|\Theta| > \frac{1}{\eps^b} \sqrt{\left( \frac{\bar{\delta}}{2} - \eps^{2b} {\left(\frac{\lambda}{\beta}\right)}^2 \left| 2-\frac{1}{\beta}\right|- a\right)}\right| \xi^2 \eps^{2b} \Jj_1 < a \right) \\
&\le \frac{\PP\left(|\Theta| > \frac{1}{\eps^b} \sqrt{\left( \frac{\bar{\delta}}{2} - \eps^{2b} {\left(\frac{\lambda}{\beta}\right)}^2 \left| 2-\frac{1}{\beta}\right|- a\right)}\right)}{\PP(\xi^2 \eps^{2b} \Jj_1 < a)}
\le \frac{\exp\left(-\frac{M}{\eps^{2b}}\right)}{1-\eta}.
\end{aligned}
$$
Finally, for all $\eps \le \min \{ \tilde{\eps}, \bar{\eps} \}$,
\begin{equation}\label{eq:bound_Yinfty}
\begin{aligned}
\PP\left(\Jj > \bar{\delta} \right)
&\le \PP\left(\left.|\Theta| > \frac{1}{\eps^b} \sqrt{\left( \frac{\bar{\delta}}{2} - \eps^{2b} {\left(\frac{\lambda}{\beta}\right)}^2 \left| 2-\frac{1}{\beta}\right|- a\right)}\right| \xi^2 \eps^{2b} \Jj_1 < a \right) + \PP \left( \xi^2 \eps^{2b} \Jj_1 \ge a \right),\\
&\le \frac{\E^{-M/\eps^{2b}}}{1-\eta} 
+ \frac{\xi^2 \eps^{2b}}{a} \frac{1}{2\Hp},
\end{aligned}
\end{equation}
so that for all $\eps \le \min \{ \tilde{\eps}, \bar{\eps} \}$, $\PP \left( 2C^2 \Jj \ge a_0 \right) \le \frac{\E^{-M/\eps^{2b}}}{1-\eta} 
+ \frac{\xi^2 \eps^{2b}}{a} \frac{1}{2\Hp}$.

Hence, there exists $\breve{\eta} \in [0,1)$ and $\breve{\eps}< \min \{ \tilde{\eps}, \bar{\eps} \}$ such that for all $\eps \le \breve{\eps}$, $\PP(2C^2 \Jj < a_0) \ge 1-\breve{\eta}$.
Moreover, using the uniform convergence assumed in~$\Aa$, for all $y\in\RR$ there exists $N>0$ and $\eps_0 >0$ such that, for all $\eps < \eps_0$, we have that $\left| \eps^b \sigma \left( \frac{Y^\eps_s}{\eps^b} \right) - \widetilde{\sigma} (Y^\eps_s) \right| \le N$. Hence one can then apply the Borell-TIS inequality (in particular~\cite{P07}[Proposition A.1]), and obtain,
$$
\PP \left( \left\| \int_0^\cdot \left( \eps^b \sigma \left( \frac{Y^\eps_s}{\eps^b} \right) - \widetilde{\sigma} (Y^\eps_s) \right) \D (\rho B_s + \overline{\rho} B^\perp_s) \right\|_{\infty} > \frac{1}{\eps^b} \left(\delta - C^2(1+\eps^{2b})-a_0  \right) \right)
\le 2\exp \left\{- \frac{\left(\delta - C^2(1+\eps^{2b})-a_0  \right)^2}{2N^2  \eps^{2b}}\right\}. 
$$

Thus, for all $\eps \le \min \{ \breve{\eps}, \eps_0\}$, Bayes' Theorem yields
$$
\begin{aligned}
&\PP\left(\left. \left\| \int_0^\cdot \left( \eps^b \sigma \left(\frac{Y^\eps_s}{\eps^b} \right) - \widetilde{\sigma} (Y^\eps_s) \right) \D (\rho B_s + \overline{\rho} B^\perp_s) \right\|_\infty > \frac{1}{\eps^b} \left(\delta - C^2(1+\eps^{2b})-a_0  \right) \right| 2C^2 \Jj < a_0 \right) \\
&\le \frac{\PP\left( \left\| \int_0^\cdot \left( \eps^b \sigma \left(\frac{Y^\eps_s}{\eps^b} \right) - \widetilde{\sigma} (Y^\eps_s) \right) \D (\rho B_s + \overline{\rho} B^\perp_s) \right\|_\infty > \frac{1}{\eps^b} \left(\delta - C^2(1+\eps^{2b})-a_0  \right) \right)}{\PP(2C^2 \Jj < a_0)}, \\
&\le \frac{2}{1-\breve{\eta}}\exp \left\{- \frac{\left(\delta - C^2(1+\eps^{2b})-a_0  \right)^2}{2N^2  \eps^{2b}}\right\}.
\end{aligned}
$$

Finally, for all $\eps \le \min \{ \breve{\eps}, \eps_0\}$,
$$
\PP\left(\left\|X^\eps - \widetilde{X}^\eps \right\|_{\infty} > \delta \right)
\le \frac{2\exp \left\{- \frac{\left(\delta - C^2(1+\eps^{2b})-a_0  \right)^2}{2N^2 \eps^{2b}}\right\}}{1-\breve{\eta}} + \frac{\E^{-M/\eps^{2b}}}{1-\eta} + \frac{\xi^2 \eps^{2b}}{a} \frac{1}{2\Hp}, 
$$
so that $X^\eps$ and $\widetilde{X}^\eps$ are exponentially equivalent and $X^\eps\sim\LDP(\eps^{2b}, \widetilde{\Lambda})$.

\begin{remark}
For computational purposes, such infinite-dimensional optimisation problems are discretised over a set of basis functions, such as orthonormal polynomials.
If we then consider~$\varphi$ continuously differentiable and~$I_1$ finite,
we can simplify 
$I_1(\varphi) = I_1(\varphi, \psi) = \frac{1}{2\rho^2}\int_{\Tt}[\dot{\varphi}(t)]^2\D t$.
Further, if $u\in \Cc^1$, then 
$I_3(\varphi)\simeq \frac{1}{2}\inf_{u\in \Hh\cap \Cc^1}\left\{\int_0^t \left(\frac{\dot{u}_s}{\rho}\right)^2\D s+ \frac{1}{\overline{\rho}}\int_0^t (\dot{\varphi}_s-\dot{u}_s)^2 \D s\right\}$, and
$I_4(\varphi) = I_3(\varphi)$. 
\end{remark}

%%%%%%%%%%%%%%%%%%%%%%%%%%%%%%%%%%%%%%%%%
%%%%%%%%%%%%%%%%%%%%%%%%%%%%%%%%%%%%%%%%%
\subsection{Proof of Theorem~\ref{Th:LDP_ST_X}}\label{sec:Th:LDP_ST_X}
\subsubsection{Proof of Theorem~\ref{Th:LDP_ST_X}(i)}
Let us introduce the process $\overline{Y}^{\eps}$ defined, for $t\in \Tt^*$, as the unique solution of 
$\D \overline{Y}^{\eps}_t = (\eps^{b+2} \lambda + \beta \eps^2 \overline{Y}^{\eps}_t) \D \tau + \eps^{2H+b} \xi \D W^H_t$, with $\overline{Y}^{\eps}_0 = 0$, i.e.
$$
\overline{Y}^{\eps}_t = - \frac{\lambda \eps^b}{\beta} \left( 1-\E^{\beta \eps^2 t} \right) + \xi \eps^{2H+b} \int_{0}^{t} \E^{\beta \eps^2 (t-u)} \D W^H_u.
$$
From Lemma~\ref{lem:rep_fBm}, Proposition~\ref{Pp:RKHS_ST} and Theorem~3.4.5 in~\cite{DBookS}, 
the sequence ${\left(\eps^{2H+b} \int_{0}^{\cdot} G^H_0 (\cdot,s) \D Z_s\right)}_{\eps >0}$ 
satisfies a large deviations principle on~$\Cc$, with speed $\eps^{4H+2b}$ and good rate function $\Lambda_{G^H_0}$ as in~\eqref{eq:LDPVolterrageneral}. Moreover,
\begin{itemize}
\item[(1)] When $H \neq \frac{1}{2}$, ${(\eps^{2H+b} \int_0^t G^H_\eps (t,s) \D Z_s)}_{\eps >0}$ 
and ${(\eps^{2H+b} \int_0^t G^H_0 (t,s) \D Z_s)}_{\eps >0}$ are exponentially equivalent:
indeed the asymptotic expansion of the Gaussian density~\cite[Formula (26.2.12)]{AS72}
yields, for any $\delta >0$,
\begin{equation*}
\begin{array}{rl}
 & \displaystyle \PP \left(\left\|\eps^{2H+b} \int_0^\cdot G^H_\eps (\cdot,s) \D Z_s -\eps^{2H+b} \int_0^\cdot G^H_0 (\cdot,s) \D Z_s\right\|_\infty > \delta\right)\\
 & = \displaystyle \PP \left(\left\|\int_0^\cdot (G^H_\eps (\cdot,s)-G^H_0 (\cdot,s))\D Z_s\right\|_\infty
 > \frac{\delta}{\eps^{2H+b}}\right)
 = \displaystyle \PP \left(|\Nn(0,1)| > \frac{\delta}{\left\|V_\eps\right\|_\infty \eps^{2H+b}}\right), \\
 & = \displaystyle \sqrt{\frac{2}{\pi}} \frac{\left\|V_\eps\right\|_\infty  \eps^{2H+b}}{\delta} \exp \left(- \frac{\delta^2}{2\eps^{4H+2b}\left\|V_\eps\right\|^2_\infty}\right) \left(1+\Oo(\eps^{4H+2b})\right),
\end{array}
\end{equation*}
with $V^2_\eps := \int_0^t {\left( G^H_\eps (t,s) - G^H_0(t,s) \right)}^2 \D s$
and note that $\lim_{\eps \downarrow 0} V^2_\eps =0$.
Now, for $\eps >0$ and $0<s<t$,
\begin{itemize}

\item[(i)] if $H < \frac{1}{2}$,
\begin{align*}
\left[ G^H_\eps (t,s) - G^H_0(t,s)\right]^2
&\le \frac{2\kappa^2_H\xi^2}{s^{2\Hm}}
\left\{\Hm^2 
\left( \int_s^t \frac{(u-s)^{\Hm}}{u^{1-\Hm}} (\E^{\beta \eps^2 (t-u)}-1) \D u \right)^2 
  + \beta^2 \eps^4 {\left( \int_s^t (u-s)^{\Hm} u^{\Hm} \E^{\beta \eps^2 (t-u)} \D u\right)}^2\right\}\\
&\le 2\kappa^2_H \xi^2 \left[\frac{\Hm^2}{s^{2\Hm}} \left( \int_s^t \frac{(u-s)^{\Hm}}{u^{1-\Hm}} \D u \right)^2
 + \frac{\beta ^2 \eps^4}{s^{2\Hm}} \left( \int_s^t {(u(u-s))}^{\Hm} \D u\right)^2 \right].
\end{align*}
Using~\cite[Lemma A.3]{YLX08}, we further obtain, for $\eps >0$ and $0<s<t$, $V^2_\eps 
\le 2\kappa^2_H \xi^2 \left[ C_H t^{2H} + \beta^2 \eps^4 \widetilde{C}_H t^{2H+2} \right]$,
with $C_H, \widetilde{C}_H >0$ two constants depending on~$H$. 
Hence, 
$$
0<\eps^{4H+2b} V^2_\eps \le 2 \kappa^2_H\xi^2 \left( C_H \eps^{4H+2b} t^{2H} + \beta^2 \eps^{4H+2b+4}\widetilde{C}_H t^{2H+2} \right);
$$
\item[(ii)] if $H > \frac{1}{2}$,
\begin{align*}
\left[ G^H_\eps (t,s) - G^H_0(t,s)\right]^2
&= \frac{\kappa^2_H\xi^2}{s^{2\Hm}} \Hm^2 {\left( \int_s^t \frac{u^{\Hm}}{(u-s)^{1-\Hm}} (\E^{\beta\eps^2 (t-u)}-1) \D u \right)}^2\\
&\le \frac{\kappa^2_H\xi^2}{s^{2\Hm}} \Hm^2 {\left( \int_s^t \frac{u^{\Hm}\D u}{(u-s)^{1-\Hm}}\right)}^2
\le \frac{\kappa^2_H\xi^2}{s^{2\Hm}} \Hm^2 t^{2\Hm} {\left( \int_s^t (u-s)^{\Hm-1} \D u \right)}^2
= \frac{\kappa^2_H\xi^2}{s^{2\Hm}} [t(t-s)]^{2\Hm}.
\end{align*}
\end{itemize}
Hence, for $H \neq \frac{1}{2}$, as $b>0$, $\lim_{\eps \downarrow 0} \eps^{4H+2b} \left\|V_\eps\right\|^2_\infty = 0$ as well as $\lim_{\eps \downarrow 0} \eps^{2H+b} \left\|V_\eps\right\|_\infty = 0$.

\item[(2)] When $H=\frac{1}{2}$,
${(\xi\eps^{1+b} \int_0^t \E^{\beta \eps^2 (t-s)} \D B_s)}_{\eps >0}$ 
and ${(\xi\eps^{1+b} B_t)}_{\eps >0}$ are exponentially equivalent:
the asymptotic expansion of the Gaussian density near infinity~\cite[Formula (26.2.12)]{AS72} yields, 
for any $\delta >0$,
\begin{equation*}
\begin{array}{rl}
\displaystyle \PP\left(\left\|\xi\int_0^\cdot \E^{\beta \eps^2 (\cdot-s)} \D B_s - \xi\int_0^\cdot \D B_s\right\|_\infty > \frac{\delta}{\eps^{1+b}}\right)
 & = \displaystyle \PP \left(\left\|\int_0^\cdot \xi(\E^{\beta \eps^2 (\cdot-s)} -1)\D B_s\right\|_\infty > \frac{\delta}{\eps^{1+b}}\right)\\
 & = \displaystyle \PP \left(|\Nn(0,1)| > \frac{\delta}{\left\|V_\eps\right\|_\infty \eps^{1+b}\xi}\right) \\
 & = \displaystyle \sqrt{\frac{2}{\pi}} \frac{\left\|V_\eps\right\|_\infty \eps^{1+b}\xi}{\delta} \exp \left(- \frac{\delta^2}{2\eps^{2+2b}\left\|V_\eps \right\|^2_\infty\xi^2}\right) \left(1+\Oo(\eps^{2+2b})\right),
\end{array}
\end{equation*}
with $V^2_\eps := \int_0^t {\left[(\E^{\beta \eps^2 (t-s)} -1) \right]}^2 \D s
= t + \frac{\E^{2\beta\eps^2 t}-1}{2\beta\eps^2}+\frac{2(1-\E^{\beta\eps^2 t})}{\beta\eps^2}$.
Hence
$
0<\eps^{2+2b} V^2_\eps 
\le \eps^{2+2b} t + \eps^{2b}\frac{\E^{\beta\eps^2 t}-1}{2\beta},
$
and, as $b>0$, $\lim_{\eps \downarrow 0} \eps^{2+2b} \left\|V_\eps \right\|^2_\infty = 0$. 
Besides, note that $\lim_{\eps \downarrow 0} V^2_\eps =0$.
The required exponential equivalence then follows, and
hence for all $H \in (0,1)$,
${(\eps^{2H+b} \int_{0}^{\cdot} G^H_\eps (\cdot,s) \D Z_s)}_{\eps >0}\sim\LDP(\eps^{4H+2b},\Lambda_{G^H_0})$.

The final step is exponential equivalence between $\overline{Y}^{\eps}$ 
and ${(\eps^{2H+b} \int_{0}^{\cdot} G^H_\eps (\cdot,s) \D Z_s)}_{\eps >0}$. 
For any $\delta,t >0$,
\begin{equation*}
\begin{array}{ll}
\displaystyle \PP\left(\left\|\overline{Y}^{\eps} - \eps^{2H+b} \int_0^\cdot G^H_\eps (\cdot,s) \D Z_s\right\|_\infty>\delta\right)
 & = \displaystyle \PP\left(\left\|\overline{Y}^{\eps} - \eps^{2H+b}  \int_{0}^{\cdot} \E^{\beta \eps^2 (\cdot-s)} \D W^H_s\right\|_\infty>\delta\right)\\
 & = \displaystyle \PP\left(\left\|- \frac{\lambda \eps^b}{\beta} (1-\E^{\beta \eps^2 \cdot})\right\|_\infty>\delta\right).
\end{array}
\end{equation*}
Hence $\limsup_{\eps \downarrow 0} \eps^{4H+2b} \log \PP(\left\|\overline{Y}^{\eps} - \eps^{2H+b} \int_0^\cdot G_\eps (\cdot,s) \D Z_s\right\|_\infty>~\delta)=- \infty$, yielding exponential equivalence and
$\overline{Y}^{\eps}\sim\LDP(\eps^{4H+2b},\Lambda_{G_0^H})$.
Finally, since for any $\delta>0$, 
$$
\PP\left(\left\|Y^{\eps}-\overline{Y}^{\eps}\right\|_\infty> \delta\right)
\le \PP \left( \eps^b \left|\Theta\right| > \delta \right),
$$
as $\beta <0$, and
$\overline{Y}^{\eps}$ and $Y^{\eps}$ are exponentially equivalent, and $Y^{\eps}\sim\LDP(\eps^{4H+2b},\Lambda_{G^H_0})$
by Assumption~$\Abt$.
\end{itemize}

%%%%%%%%%%%%%%%%%%%%%%%%%%%%%%%%%%%%%%%%%
%%%%%%%%%%%%%%%%%%%%%%%%%%%%%%%%%%%%%%%%%
\subsubsection{Proof of Theorem~\ref{Th:LDP_ST_X} (ii)}
The idea of the proof is the same as the proof of Theorem~\ref{Th:tails_frac_X}(ii). 
\begin{itemize}
\item[(1)] $H \neq\frac{1}{2}$:
we already established that $\eps^{2H+b}(\overline{\rho}B^\perp + \rho B, W^H)$ 
satisfies a large deviations principle on $\Cc(\Tt,\RR^2)$, with speed $\eps^{4H+2b}$ and good rate function $I_3$ defined earlier.

\item[(2)] $H = \frac{1}{2}$:
using the proof of Theorem~\ref{thm:ext_MNSLDP}, $\eps^{b+1} (\overline{\rho}B^\perp + \rho B,B)'$  
satisfies a large deviations principle on $\Cc(\Tt,\RR^2)$, with speed $\eps^{2+2b}$ 
and good rate function $\Lambda^{\rho}$ defined in~\eqref{eq:ext_grf}.
\end{itemize}

Similar to above, we only prove the case $H\ne\frac{1}{2}$, the other one being analogous. We introduce the process $\widetilde{X}^\eps$ satisfying the following SDE $\D \widetilde{X}^\eps_t = -\frac{\eps^{2H+1}}{2} \widetilde{\sigma}^2(Y^\eps_t) \D t + \eps^{2H+b} \widetilde{\sigma} (Y^\eps_t) (\rho \D B_t + \overline{\rho} \D B^\perp_t)$,
and translate this large deviations into one for the sequence~$\widetilde{X}^\eps$.

Using that for $\eps >0$ and $t\in\Tt^*$, 
$Y^{\eps}_t = \eps^b \E^{\beta \eps^2 t} \Theta - \frac{\lambda \eps^b}{\beta}(1-\E^{\beta \eps^2 t}) + \eps^{2H+b} \xi \int_0^t \E^{\beta \eps^2 (t-u)} \D W^H_u$, one can define a continuous map 
$\widetilde{G}((0,1)\times\Cc\to\Cc)$, such that
$(\eps^{2H+b} (\rho B + \overline{\rho} B^{\perp}), Y^{\eps}) = (\eps^{2H+b} (\rho B + \overline{\rho} B^{\perp}), \widetilde{G}(\eps, \eps^{2H+b} W^H)(t))$. 
Using $\Abt$ and the asymptotic expansion of the Gaussian density near infinity~\cite[Formula (26.2.12)]{AS72}, $\widetilde{G}(\eps, \eps^{2H+b} W^H)$ is exponentially equivalent to $\widehat{G}_{0}(\eps^{2H+b} W^H)$, defined by $\widehat{G}_{0}(\varphi)(t) := \xi \int_0^t \D \varphi_u$.

Hence, the Contraction Principle yields an LDP on $\Cc(\Tt,\RR^2)$ for $(\eps^{2H+b} (\rho B + \overline{\rho} B^{\perp}), Y^{\eps})$, with speed $\eps^{4H+2b}$ and good rate function
 $\widetilde{I}_4: (\varphi, \psi) \mapsto \inf \{ I_3(\varphi,v) \ | \ \psi \in \Cc : \psi=\widehat{G}_{0}(v) \}$.
Under~$\Aa'$, the extended Contraction Principle~\cite[Proposition 2.3]{MNP92}
yields an LDP on $\Cc(\Tt,\RR^2)$ for $(\eps^{2H+b} (\rho B + \overline{\rho} B^\perp), \widetilde{\sigma}(Y^{\eps}))$, with speed $\eps^{4H+2b}$ and good rate function 
$\widetilde{I}_5:(\varphi, \psi) \mapsto \inf \left\{ \widetilde{I}_4(\varphi,v) \ | \ \psi=\widetilde{\sigma}(v) \right\} = \inf \left\{ I_3(\varphi,v) \ | \ \psi=\widetilde{\sigma}(\widehat{G}_{0}(v)) \right\}$.
Finally, setting $b \ge \frac{1}{2}-2H$, the sequence of semi-martingales $(\eps^{2H+b} W)$ is UET 
and the sequence $(\widetilde{\sigma}(Y^\eps))$ is c\`adl\`ag (Assumption~$\Aa$), 
and adapted to the filtration~$\Ff$. 
Theorem~1.2 in~\cite{Gar08} thus gives an LDP on~$\Cc$ for 
the sequence of stochastic integrals $\eps^{2H+b} \int_0^\cdot \widetilde{\sigma}(Y^\eps_s) (\rho \D B_s + \overline{\rho} \D B^\perp_s)$, 
with speed $\eps^{4H+2b}$ and good rate function 
\begin{equation} \label{eq:grf_X_ST}
\II(\chi):= \inf \left\{ \widetilde{I}_5 (\varphi,\psi) : \varphi \cdot \psi=\chi, \ \psi\in\BV\right\}.
\end{equation}

We now prove that $\widetilde{X}^\eps$ and $\eps^{2H+b} \int_0^\cdot \widetilde{\sigma}(Y^\eps_s) (\rho \D B_s + \overline{\rho} \D B^\perp_s)$ are exponentially equivalent.

$$
\begin{aligned}
\PP\left(\left\| \widetilde{X}^\eps-\eps^{2H+b} \int_0^\cdot \widetilde{\sigma}(Y^\eps_s) \D W_s\right\|_\infty > \delta\right) 
&= \PP\left(\left\|-\frac{1}{2} \int_0^\cdot {(\eps^{H+\frac{1}{2}}\widetilde{\sigma}(Y^\eps_s))}^2 \D s\right\|_\infty > \delta\right)
= \PP\left(\left\|\int_0^\cdot (\widetilde{\sigma}(Y^\eps_s))^2 \D s \right\|_\infty > \frac{2\delta}{\eps^{2H+1}}\right), \\
& \le \PP\left(\left\|\int_0^\cdot |Y^\eps_s|^2 \D s \right\|_\infty > \frac{\delta}{C^2 \eps^{2H+1}}-1 \right),
\end{aligned}
$$
using the growth condition assumption $|\widetilde{\sigma}(x)| \le C(1+|x|)$ for $x \in \RR$.
Since $Y^\eps_s = \eps^b \E^{\beta \eps^2 s} \Theta + \eps^b \frac{\lambda}{\beta} (\E^{\beta \eps^2 s}-1) + \eps^{2H+b}\xi \left[ W^H_s +\beta \eps^2 \int_0^s W^H_u \E^{-\beta \eps^2 u} \D u \right]$
and $\beta<0$, 
we obtain for $s \in \Tt$,
$$
|Y^\eps_s|
\le \eps^b \E^{\beta \eps^2 s} |\Theta|
+ \eps^b \frac{\lambda}{\beta} \left(\E^{\beta \eps^2 s}-1\right) 
+ \xi \eps^{2H+b} |W^H_s|,
$$
so that
$$
|Y^\eps_s|^2
\le 2\left(\eps^{2b} \E^{2\beta \eps^2 s} \Theta^2 
+ 2\eps^{2b} {\left(\frac{\lambda}{\beta}\right)}^2 \left(\E^{2\beta \eps^2 s}+1\right) 
+ \xi^2 \eps^{4H+2b} |W^H_s|^2\right), 
$$
and hence
$\int_0^t |Y^\eps_s|^2 \D s 
\le \eps^{2b} t \Theta^2 +\eps^{2b} {\left(\frac{\lambda}{\beta}\right)}^2 \left(2t-\frac{1}{\beta\eps^2}\right) + \xi^2 \eps^{4H+2b} \int_0^t |W^H_s|^2 \D s$.
Using a similar argument to~\eqref{eq:bound_Yinfty}, used in the proof of Theorem~\ref{Th:tails_frac_X}, we have, for all $\eps \le \min \{ \tilde{\eps}, \bar{\eps} \}$,

\begin{equation*}
\PP\left(\left\|\int_0^\cdot |Y^\eps_s|^2 \D s \right\|_\infty > \frac{\delta}{C^2 \eps^{2\Hp}}-1 \right)
\le \frac{\E^{-M/\eps^{4H+2b}}}{1-\eta} 
+ \frac{\xi^2 \eps^{4H+2b}}{a} \frac{1}{2\Hp},
\end{equation*}
so that $\widetilde{X}^\eps$ and 
$\eps^{2H+b} \int_0^\cdot \widetilde{\sigma}(Y^\eps_s) (\rho \D B_s + \overline{\rho} \D B^\perp_s)$ 
are exponentially equivalent.

Finally, one can use the same arguments that in Appendix~\ref{Sec:Proof1}, replace $\beta$ by $\beta \eps^2$ and the speed $\eps^{2b}$ by $\eps^{4H+2b}$ and prove that $X^\eps$ and $\widetilde{X}^\eps$ are exponentially equivalent. Hence $X^\eps \sim \LDP (\eps^{4H+2b}, \II)$.

%%%%%%%%%%%%%%%%%%%%%%%%%
%%%%%%%%%%%%%%%%%%%%%%%%%
\subsection{Proof of Theorem~\ref{Th_rdm_new}}\label{sec:Proof-Th_rdm_new}
The proof is similar to that of~\cite[Theorem 2.9]{P07}, with some minor modifications.
One key ingredient in the proof of the theorem is an exponential equivalence 
between~$(\X^{\eps,\xx})_{\eps>0}$ and its limit system (as~$\eps$ tends to zero). In this proof, we denote $\X^\eps$ by $\X^{\eps,\xx}$  to stress the dependence on the starting point.
\begin{lemma}\label{Lem:main_est}
Under Assumption~\ref{ass:UniqueSol},
for each $R, \delta, \beta>0$, there exists $\gamma, \rho, \eps_0 >0$ 
such that
$$
\eps^2 \log \PP\left(\|\X^{\eps,\xx}-\Ss_{0}(\varphi,u)\|>\delta, \|\eps \W - \varphi\| \le \gamma\right) \le -R
$$
holds when $\eps \le \eps_0$
for all $u \in supp(\X_0)$ and all $\varphi\in \Hh$ satisfying $\Lambda^{\rho}(\varphi) \le \beta$, 
$\xx \in \Bb_\rho ((0,0))$.
\end{lemma}

\begin{proof}
First, introduce the Radon-Nikodym derivative
$\mathrm{D}^\eps(\varphi):=\exp \left\{\frac{1}{\eps} \int_{\Tt} \dot{\varphi}_s \D \W_s - \frac{1}{2\eps^{2}} \int_{\Tt} {\|\dot{\varphi}_s\|}^2 \D s\right\}$.
Following~\cite{P07}, consider the family ${(\overline{\Y}^\eps)}_{\eps >0}$ of (unique strong ) solutions of 
$\D \overline{\Y}^\eps_t = \ck(\eps,t,\overline{\Y}^\eps_t,\X_0) \D t + \eps \ak(\overline{\Y}^\eps_t,\X_0) \D \mathrm{B}^\eps_t$, for $t >0$, 
with initial condition $\overline{\Y}^\eps_0 = (0,0)' \in \RR^2$.
We denote $\ck(\eps,t, \xx, \yy):= \bk(\eps,\xx,\yy) + \ak(\xx,\yy) \dot{f}_t$, for $\xx, \yy \in \RR^2$, $t \in \Tt$ and $\eps >0$, and $\mathrm{B}^\eps_t := W_t - \eps^{-1} \dot{f}_t$, for $t \in \Tt$.
Peithmann's assumptions are here updated as:
\begin{itemize}
\item[(i)] $\ck:(0,\infty)\times\Tt\times\RR^2\times\RR^2 \rightarrow \RR^2$ converges to $\ck_0:\Tt\times\RR^2\times\RR^2 \rightarrow \RR^2$ in the sense that 
$$
\lim_{\eps \downarrow 0} \int_{\Tt} \sup_{\xx\in \RR^2} \|\ck(\eps,t,\xx,\yy) - \ck_0(t,\xx,\yy)\|\D t = 0,
\qquad\text{for all }\yy\in \RR^2.
$$
In particular, note that $\ck_0(t, \xx, \yy) = \ak(\xx,\yy)\dot{f}_t$;
\item[(ii)] there exists $\varrho\in \LTt$ such that for all $\xx, \yy\in \RR^2$, 
$\|\ck(\eps,t,\xx,\yy)\| + \|\ck_0(t,\xx,\yy)\| \le \varrho(t)$, for $t \in \Tt$;
\item[(iii)] there exists $\kappa \in \mathrm{L}^1(\Tt)$ such that for all $\yy, \xx_1,\xx_2 \in \RR^2$, 
$\|\ck_0(t,\xx_1,\yy) - \ck_0(t,\xx_2,\yy)\| \le \kappa (t) \|\xx_1 - \xx_2\|$ on~$\Tt$.
\end{itemize}
Partitioning $\Tt$ into $\{t_k := \frac{kT}{n}\}_{k=0,\ldots,n}$, 
set $\overline{\Y}^{\eps,n}_t := \overline{\Y}^\eps_{t_k}$ for $t_k \le t < t_{k+1}$, 
and Peithmann's arguments yield
\begin{enumerate}[(i)]
\item for any $\delta >0$, $\lim\limits_{n \uparrow \infty} \limsup_{\eps \downarrow 0} \eps^2 \log \PP (\|\overline{Y}^\eps- \overline{\Y}^{\eps,n}\| >\delta)=-\infty$ 
uniformly with respect to $\yy_0 \in \RR^2$;
\item given $M^\eps_t := \eps \int_{\Tt} \ak(\overline{\Y}^\eps_s,\X_0)\D \W_s$ and 
$M^{\eps,n}_t := \eps \int_{\Tt} \ak(\overline{\Y}^{\eps,n}_s,\X_0)\D \W_s$, for all $\delta >0$,
$$ 
\lim_{\beta \downarrow 0} \limsup_{\eps \downarrow 0} \eps^2 \log 
\PP\left(\|M^\eps-M^{\eps,n}\| > \delta, \|\Y^\eps-Y^{\eps,n}\|\le \beta\right) = -\infty, 
$$
uniformly with respect to $n \in \mathbb{N}$, $y_0 \in \RR^2$;
\item for all $\delta >0$, 
$\lim_{\gamma \downarrow 0} \limsup_{\eps \downarrow 0} \eps^2 \log \PP(\|M^\eps\|>\delta, \|\eps\W\| \le \gamma)=-\infty$;
\item let $\zeta$ denote the solution of the ordinary differential equation 
$\dot{\zeta}_t = \ck_0(t,\zeta_t,Y_t)$ starting from $\zeta_0=\yy$. 
For all $R>0$, $\delta >0$, there exists $\gamma, \rho,\eps_0 >0$ 
such that $\PP(\left\{ \|\Y^\eps,\zeta\|>\delta \right\} \cap \left\{ \|\eps \W\| \le \gamma \right\})
 \le \exp \left( -R / \eps^2 \right)$, 
 for all $\yy_0 \in\RR^2$, $\yy\in \Bb_\rho (\yy_0)$ and $\eps \le \eps_0$.
\end{enumerate}
The theorem follows from Girsanov's theorem with Radon-Nikodym derivative~$\mathrm{D}^\eps(f)$.
\end{proof}
We start with the lower bound. 
For any open subset~$G$ of~$\Cc(\Tt,\RR^2)$, let $\eta >0$ and choose $\psi\in G$ 
such that $I^\alpha(\psi) \le \inf_{\psi\in G} I^\alpha(\psi) + \eta$. 
Then, let $u \in \supp(\X_0)$ and $\varphi \in \Hh$ such that $\Ss_0(\varphi,u)=\psi$ 
and $\Lambda^{\rho}(\varphi)=I^\alpha(\psi)$. 
Let $\delta >0$ such that $\Bb_\delta(\psi) \subset G$. 
Then, for each $\gamma >0$, $\xx \in \RR^2$, 
\begin{align*}
\PP(\X^{\eps,\xx} \in G)
 \ge \PP(\|\X^{\eps,\xx}-\psi\| \le \delta)
 \ge \PP(\|\eps \W - \varphi\| \le \gamma) - \PP(\|\X^{\eps,\xx}-\psi\|>\delta, \|\eps \W - \varphi\| \le \gamma).
\end{align*}
Schilder's Theorem~\cite[Theorem 5.2.3]{DZ}, then yields the lower bound
$$
\liminf_{\eps \downarrow 0} \eps^2 \PP(\|\eps \W - \varphi\| \le \gamma) \ge -\Lambda^{\rho}(\varphi)
 = -I^\alpha(\psi) \ge - \inf_{\psi \in G} I^\alpha(\psi) - \eta.
$$
Then, we bound the second probability from above using Lemma~\ref{Lem:main_est}: 
fix $\beta \ge \Lambda^{\rho}(\varphi)$ and $R > \inf_{\psi \in G} I^\alpha(\psi) + \eta$, 
and find $\gamma, \rho, \eps_0 >0$ such that, for $\xx \in \Bb_\rho ((0,0)')$, $\eps \le \eps_0$,
the bound 
$\eps^2 \log \PP(\|\X^{\eps,\xx} - \psi\| > \delta, \|\eps \W - \varphi\| \gamma) \le -R$ holds.
These two bounds then imply the required lower bound:
$$
\lim_{\eps \downarrow 0} \inf_{\rho \downarrow 0} \eps^2 \log \inf_{\xx \in \Bb_\rho ((0,0)')} \PP\left(\X^{\eps,\xx} \in G\right)
 \ge \min \left\{ -R, -\inf_{\psi\in G} I^\alpha(\psi) - \eta\right\}
  = -\inf_{\psi \in G} I^\alpha(\psi) - \eta.
$$

We now prove the upper bound.
For any closed set~$F$ of~$\Cc(\Tt, \RR^2)$, take $\beta \in (0,\inf_{\psi\in G} I^\alpha(\psi))$ and 
$R > \beta$. 
Let $u \in \supp(\X_0)$ and $\psi\in \Hh$ with $I^\alpha(\psi) \le \beta$. 
We find $\delta >0$ such that $\Bb_{\delta}(\psi) \cap F = \emptyset$ 
and $\varphi \in \{\Lambda^{\rho} \le \beta \}$ such that $\Ss_0 (\varphi,u)=\psi$. 
Using Lemma~\ref{Lem:main_est}, there exists $\gamma, \rho, \eps_0 >0$ such that
for $\xx \in \Bb_\rho ((0,0)')$, $\eps \le \eps_0$,
$$
\eps^2 \log \PP(\|\X^{\eps,\xx} - \psi\| > \delta, \|\eps \W - \varphi\| \le \gamma) \le -R.
$$
The set $\left\{ \Bb_\gamma (\varphi) : \psi \in \Hh, I^\alpha(\psi) \le \beta \right\}$ 
forms a cover of the compact set
$\{\Lambda^{\rho}(\varphi) \le \beta \}$, 
so that we can extract a finite sub-cover $\left\{ \Bb_{\gamma_i} (\varphi_i)\right\}_{i=1, \ldots, k}$ 
and set $A := \cup_{i=1}^{k} \Bb_{\gamma_i} (f_i)$ and $\psi_i := \Ss_0 (\varphi_i,r)$. 
For any $i=1,\ldots,k$, there exist $\delta_i, \rho_i, \eps_i>0$ such that, 
for any $\xx \in \Bb_{\rho_i}((0,0)')$ and $\eps \le \eps_i$, 
$$
\eps^2 \log \PP(\|\X^{\eps,\xx} - \psi_i\| > \delta, \|\eps \W - \varphi_i\| \le \gamma) \le -R.
$$
Set $\eps_0 := \min \left\{ \eps_1, \ldots, \eps_k \right\}$,
$\rho_0 := \min \left\{ \rho_1, \ldots, \rho_k \right\}$, 
take $\eps \le \eps_0$ and $\xx \in \Bb_{\rho_0}((0,0)')$. 
Since $F \cap \Bb_{\delta_i}(\psi_i) = \emptyset$ for every $i=1,\ldots,k$, we obtain
\begin{align*}
\PP (\X^{\eps,\xx} \in F) 
&\le \PP(\X^{\eps,\xx} \in F, \eps \W \in A) + \PP(\eps \W \in A^c)\\
&\le \sum_{i=1}^{k} \PP(\|\X^{\eps,\xx} - g_i\|> \delta_i, \|\eps \W - \varphi_i\| \le \gamma_i) + \PP(\eps \W \in A^c)
\le k\exp\left(-\frac{R}{\eps^2}\right) + \exp\left(-\frac{\beta}{\eps^2}\right),
\end{align*}
since $I^\alpha(\psi_i) \le \beta$ by definition of~$\beta$. 
Finally, since $R > \beta$, the theorem follows from the upper bound 
$$
\lim_{\eps \downarrow 0} \sup_{\rho \downarrow 0} \eps^2 \log \sup_{\xx \in \Bb_\rho (\X^\eps_0)} \PP\left(\X^{\eps,\xx} \in F\right)
 \le -\inf_{\psi\in F} I^\alpha(\psi).
$$

%%%%%%%%%%%%%%%%%%%%%%%%%
%%%%%%%%%%%%%%%%%%%%%%%%%
\subsection{Proof of Corollary~\ref{Cor:ImpliedvolTails}}\label{sec:Proof_IVTails} 
This is a straightforward application of~\cite[Corollary 7.1]{GL14}.
Taking~$\eps^{-2b}$ to be~$k$ we have, from Theorem~\ref{Th:tails_frac_X}, 
that $\eps^{2b} X_t \sim \text{LDP} (\eps^{2b}, \widetilde{\Lambda})$ as $\eps$ goes to zero. Then,
$$
\lim_{k \uparrow \infty} \frac{1}{k} \log\PP(X_t \ge k) = - \inf_{y \ge 1} \widetilde{\Lambda}(\phi)|_{\phi_t=y}.
$$ 
Similarly, in Black-Scholes, 
$\displaystyle \lim_{k \uparrow \infty} \frac{\Sigma^2_t(k)}{k^2} \log\PP (X_t \ge k)
= - \frac{1}{2t }$, and the proof follows from~\cite[Corollary 7.1]{GL14}.

%%%%%%%%%%%%%%%%%%%%%%%%%
%%%%%%%%%%%%%%%%%%%%%%%%%
\subsection{Proof of Corollary~\ref{Cor:ImpliedvolST}}\label{sec:Proof_IVST}
This is a straightforward application of~\cite[Corollary 7.1]{GL14}.
Taking $\eps^2$ to be~$t$ and $b_H := H+b-1/2$, from Theorem~\ref{Th:LDP_ST_X}, $t^{b_H} X \sim \text{LDP} (t^{2H+b}, \II)$, as~$t$ goes to zero,
$$
\lim_{t \downarrow 0} t^{2H+b} \log\PP\left(t^{b_H}X \ge k\right) = - \inf_{y \ge k} \II(\phi)|_{\phi_1=y}.
$$ 
Similarly, in the Black-Scholes model, we obtain
$$
\lim_{t \downarrow 0} \frac{t \Sigma^2_t(t^{-b_H}k)}{t^{-2b_H}} \log\PP\left(t^{b_H} X \ge k\right)
= \lim_{t \downarrow 0} \frac{t \Sigma^2_t(t^{-b_H}k)}{t^{-2b_H}} \log\PP\left(X \ge t^{-b_H} k\right)
= - \frac{k^2}{2},
$$
and the result follows from~\cite[Corollary 7.1]{GL14}.

%%%%%%%%%%%%%%%%%%%%%%%%%
%%%%%%%%%%%%%%%%%%%%%%%%%
\subsection{Proof of Proposition~\ref{prop:RSP}}\label{sec:Proof_prop_RSP}
For any $\eps >0$, the pathwise rescaled process $\X^{\eps, \zeta}$ satisfies~\eqref{eq:ScalingTails}.
The proof of the proposition relies on the (more general) theorem proved by Millet, Nualart and Sanz~\cite{MNS},
recalled in Section~\ref{app:MNS}, and whose validity is guaranteed by Assumption~\ref{ass:MNS}.
Note first that, from standard large deviations considerations (and in particular contraction mappings), 
the process~$\X^{\eps, 0}$ satisfies a large deviations principle 
with good rate function~$I$ given in Theorem~\ref{thm:MNSLDP}.
Recall that by construction
$\bk(\eps,\cdot):\RR^2\to\RR^2$ satisfies 
$\lim_{\eps\downarrow 0}\eps \widetilde{\bk}(\eps, \xx/\eps)=\bk(\xx)$ uniformly as $\eps$ tends to zero. 
Therefore Theorem~\ref{thm:MNSLDP} yields a large deviations principle for the sequence ${(\X^{\eps, \zeta})}_{\eps \ge 0}$
as~$\eps$ tends to zero, 
with good rate function~$I$ and speed~$\eps^2$.
In particular, for
$A := \left\{\psi\in \Cc(\Tt,\RR^n), \forall \xx \in \RR^n, \psi(1, \xx) \ge 1 \right\}$, 
we have
\begin{equation*}
\left\{
\begin{array}{lcccccc}
\displaystyle- \inf_{\psi\in \mathring{A}} I(\psi) & \le
 & \displaystyle \liminf_{\eps \downarrow 0} \eps^2 \log \PP(\X^{\eps, \zeta}_t \ge 1)
 & \le & \displaystyle \limsup_{\eps \downarrow 0} \eps^2 \log \PP(\X^{\eps, \zeta}_t \ge 1)
 & \le & \displaystyle - \inf_{\psi\in \bar{A}} I(\psi), \\
- \displaystyle \inf_{\psi\in \mathring{A}} I(\psi) & \le 
& \displaystyle \liminf_{\eps \downarrow 0} \eps^2 \log \PP(\X^{\eps, \xx_0}_t \ge 1) 
& \le & \displaystyle \limsup_{\eps \downarrow 0} \eps^2 \log \PP(\X^{\eps, \xx_0}_t \ge 1) 
& \le & \displaystyle - \inf_{\psi\in \bar{A}} I(\psi).
\end{array}
\right.
\end{equation*}
Since~$\Lambda$ is continuous on~$\Hh$, it is upper semi-continuous on~$A$,
and so is the good rate function~$I$ by~\cite[Lemma 2.41]{AB06}. 
As a good rate function, it is also lower semi-continuous, and hence continuous, on~$A$.
Translating the two sets of inequalities above in terms of~$\eps \X_t^{\zeta}$ and~$\eps \X_t^{\xx_0}$,
and using the continuity of~$I$ proves the proposition.

%%%%%%%%%%%%%%%%%%%%%%%%%
%%%%%%%%%%%%%%%%%%%%%%%%%
%%%%%%%%%%%%%%%%%%%%%%%%%%%%%%%%%%%%%%%%%%%%%%%%
%%%%%%%%%%%%%%%%%%%%%%%%%%%%%%%%%%%%%%%%%%%%%%%%

\end{document}